\documentclass{tgacpaper}

\title{A Cartan-Hadamard type result for relatively hyperbolic groups.}
\author{R\'emi Coulon, Michael Hull and Curtis Kent}

\begin{document}

\maketitle

\begin{abstract}
	In this article, we prove that if a finitely presented group has an asymptotic cone which is tree-graded with respect to a precise set of pieces then it is relatively hyperbolic.  
	This answers a question of M.~Sapir.
\end{abstract}

\tableofcontents

\section{Introduction}

\paragraph{}
Let $G$ be a finitely generated group, considered as a metric space with the corresponding word metric.
Roughly speaking, an asymptotic cone of $G$ is the metric space obtained by viewing $G$ from infinitely far away.
It can be defined as follows. Let $d=(d_n)$ be a sequence of real numbers diverging to infinity and $\omega$ a non-principal ultra-filter. The asymptotic cone $\ascong Gd$ is the $\omega$-limit of the sequences $(G/d_n)$ where $G/d_n$ stands for the group $G$ whose metric has been rescaled by $d_n$ (see \autoref{sec: ultra-limits} for a precise definition of ultra-limit of metric spaces). This object has been originally introduced by M.~Gromov \cite{Gromov:1981ve} and formalized later by L.~van Dries and A.~Wilkie using ultra-filters \cite{DriWil84}. An asymptotic cone captures some of the logical and geometrical properties of $G$.  In particular, it provides information about the large scale features of the group and thus serves as a quasi-isometry invariant \cite{Dru02}.  Indeed, if two groups are quasi-isometric, then their asymptotic cones are bi-Lipschitz equivalent.
In particular, the topology of the cone does not depend on the choice of a finite generating set for the group.
Asymptotic cones have been also used to characterize several classes of groups.
For instance, a finitely generated group is
\begin{enumerate}
	\item \label{enu: intro - characterization by ac - nilpotent}
	virtually nilpotent if and only if all its asymptotic cones are locally compact \cite{Gromov:1981ve,DriWil84,Dru02},
	\item \label{enu: intro - characterization by ac - hyperbolic}
	hyperbolic if and only if all its asymptotic cones are $\R$-trees \cite{Gro93}.
\end{enumerate}
Note that these statements involve \emph{all} asymptotic cones of the group.
Indeed $\ascong Gd$ does depend on the ultra-filter $\omega$ and the rescaling sequence $d$.
A finitely generated group $G$ is called \emph{lacunary hyperbolic} if \emph{one} of its asymptotic cones is an $\R$-tree \cite{OlcOsiSap09}.
Contrary to the situation described in \ref{enu: intro - characterization by ac - hyperbolic}, this class contains many non-hyperbolic groups: Gromov's Monster group \cite{Gro03,ArzDel08}, some Tarski monsters \cite{OlcOsiSap09,Olc80}, non-virtually cyclic elementary amenable groups \cite{OlcOsiSap09}, etc.
However in each of these examples the group is not finitely presented.
This is essential because of the following statement.

\begin{theo}[Kapovitch-Kleiner]{\rm \cite[Theorem 8.1]{OlcOsiSap09}}
\label{res: intro - main theorem for hyperbolic groups}
	Let $G$ be a finitely presented group.
	If \emph{some} asymptotic cone of $G$ is an $\R$-tree, then $G$ is hyperbolic.
\end{theo}

This theorem can been understood as a Cartan-Hadamard type statement for groups.
In Riemannian geometry, one challenge is to understand the global properties of a manifold using information from microscopic scale, in particular the ones supported by the curvature tensor.  The Cartan-Hadamard theorem states that the universal cover of a complete manifold with negative sectional curvature is homeomorphic to a Euclidian space. A similar local-to-global phenomenon exists for $\delta$-hyperbolic spaces (in the sense of Gromov \cite{Gro87}). However the $\delta$-hyperbolicity only reflects the metric properties of the space at a large scale compared to $\delta$. Therefore one does not want to use the microscopic scale of Riemannian geometry and prefers to look at balls of fixed radius $\sigma \gg \delta$.
A space $X$ is $\sigma$-locally $\delta$-hyperbolic, if every ball of $X$ of radius $\sigma$ is $\delta$-hyperbolic.
If any based loop in $X$ is homotopic to a product of loops each of which is freely homotopic to a loop of diameter at most $\epsilon>0$, we say that $X$ is \emph{$\epsilon$-simply connected}.
The following Cartan-Hadamard theorem has been formulated by M.~Gromov \cite{Gro87}.
There are many proofs in the literature (for example: B.~Bowditch \cite{Bowditch:1991wl}, P.~Papsoglu \cite{Papasoglu:1996uv}, or T.~Delzant and M.~Gromov \cite{DelGro08}).
We refer here to the statement given by the first author in \cite[Theorem A.1]{Coulon:2013tx}.

\begin{theo}[Cartan-Hadamard Theorem]
\label{res: intro - cartan hadamard}
Let $\delta \geq 0$ and  $\sigma >10^7 \delta$.
Let $X$ be a length space.
If $X$ is $\sigma$-locally $\delta$-hyperbolic and $10^{-5}\sigma$-simply connected then $X$ is (globally) $300\delta$-hyperbolic.
\end{theo}

From this point of view \autoref{res: intro - main theorem for hyperbolic groups} is a reformulation of the Cartan-Hadamard theorem.
Since one asymptotic cone of $G$ is an $\R$-tree, for every $\delta>0$, there exists  $\sigma > 10^7\delta$ such that $G$ is $\sigma$-locally $\delta$-hyperbolic.
On the other hand, $G$ being finitely presented, its Cayley graph is $\epsilon$-simply connected for a sufficiently large $\epsilon$.
Thus the hyperbolicity of $G$ follows from the Cartan-Hadamard theorem.

\paragraph{}
In this paper we provide a generalization of \autoref{res: intro - main theorem for hyperbolic groups} for relatively hyperbolic groups.
The notion of a group being hyperbolic relative to a collection of subgroups was introduced by Gromov in \cite{Gro87}.
It provides a natural generalization of hyperbolicity.
Indeed a group is hyperbolic if and only if it is hyperbolic relative to the trivial subgroup.
In addition, many results about hyperbolic groups have some natural analogue for relatively hyperbolic groups.
Examples of relatively hyperbolic groups include:
the fundamental group of a complete, finite volume manifold with pinched negative sectional curvature is hyperbolic relative to the cusp subgroups \cite{Bowditch:2012ga,Far98},
geometrically finite Kleinian groups are hyperbolic relative to the set of maximal parabolic subgroups \cite{Bowditch:2012ga,Yaman:2004ka},
and the fundamental group of a finite graph of groups with finite edge groups is hyperbolic relative to the vertex groups \cite{Bowditch:2012ga}.
Note that every finitely generated group with infinitely many ends is contained in this last example by the famous theorem of Stallings \cite{Stallings:1968us,Stallings:1971up}.

\paragraph{}
There are several definitions of relative hyperbolicity (see \cite{Hruska:2010iw} and references therein). We will use the following definition, which is essentially the same as the original definition of Gromov \cite{Gro87}. For simplicity, we will present in the introduction our work for a single parabolic subgroup.
However everything holds for a finite family of parabolic subgroups.

\begin{defi}{\rm \cite[Definition 3.3]{Hruska:2010iw}} \quad
\label{def: intro - relatively hyperbolic}
	Let $G$ be a group and $H$ a subgroup of $G$.
	Assume that $G$ acts properly on a proper $\delta$-hyperbolic space $X$, such that every maximal parabolic subgroup is a conjugate of $H$.	
	Suppose also that there is a $G$-equivariant collection of disjoint open horoballs centered at the parabolic points of $G$, with union $U$, such that $G$ acts co-compactly on $X \setminus U$.
	Then $G$ is \emph{hyperbolic relative to} $H$.
\end{defi}

The asymptotic cones of relatively hyperbolic groups have been described by C.~Dru\c tu, D.~Osin and M.~Sapir in \cite{DruSap05}.
To that end they introduced the notion of a \emph{tree-graded} space.
\begin{defi}
\label{def: intro - tree graded space}
	Let $X$ be a complete geodesic metric space and $\mathcal Y$ a collection of closed geodesic subspaces (called \emph{pieces}).
	The space $X$ is \emph{tree-graded} with respect to $\mathcal Y$ if the following holds.
	\begin{indexedenu}[T]
		\item Every two different pieces have at most one common point.
		\item Every simple closed curve in $X$ is contained in one piece.
	\end{indexedenu}
\end{defi}

Asymptotic cones of relatively hyperbolic groups are tree-graded with respect to limits of cosets of the parabolic subgroups \cite{DruSap05}. To be more precise,
let $G$ be a finitely generated group and let $H$ be a subgroup of $G$.
Let $\omega$ be a non-principal ultra-filter and $d=(d_n)$ a sequence of rescaling parameters diverging to infinity.
Given a collection of subsets $(Y_n)$ of $G$, the limit set $\limo Y_n$ is the set of points of $\ascong Gd$ that can be obtained as a limit of points $(y_n)$ with $y_n \in Y_n$.
The group $G$ is \emph{sparsely asymptotically tree-graded} (\resp \emph{asymptotically tree-graded}) with respect to $H$ if there exists an asymptotic cone such that (\resp for every asymptotic cone) the following holds.
\begin{indexedenu}[T^\omega]
	\item For every  sequence $(g_n)$ of elements of $G$, if $\limo g_nH$ and $\limo H$ have more than one common point then $g_n \in H$ $\omega$-almost surely.
	\item Every simple closed curve in $\ascong Gd$ is contained in some limit $\limo g_nH$ with $g_n \in G$.
\end{indexedenu}

\begin{theo}[Dru\c tu-Osin-Sapir]{\rm \cite[Theorem 8.5]{DruSap05}}\quad
\label{the: intro - asymptotic cones of relatively hyperbolic groups}
	A finitely generated group $G$ is asymptotically tree-graded with respect to a subgroup $H$ if and only if it is hyperbolic relative to $H$.
\end{theo}

In view of \autoref{res: intro - main theorem for hyperbolic groups} and \autoref{the: intro - asymptotic cones of relatively hyperbolic groups}, M.~Sapir asked the following question.
Let $G$ be a finitely presented group and $H$ a subgroup of $G$.
If $G$ is sparsely asymptotically tree-graded with respect to $H$, then is $G$ hyperbolic relative to $H$?
Our main result answers this question positively. Moreover, we show that $G$ does not have to be finitely presented, but only finitely presented \emph{relative} to the subgroup $H$.
Recall that a group $G$ is \emph{finitely presented relative to a subgroup} $H$ if there exist a finite subset $S$ of $G$ and a finite subset $R$ of the free product $\mathbf F(S) * H$ (where $\mathbf F(S)$ denotes the free group generated by $S$) such that $G$ is isomorphic to the quotient of $\mathbf F(S)*H$ by the normal subgroup generated by $R$.
In particular if $G$ is a finitely presented group and $H$ a subgroup of $G$, then $G$ is finitely presented relative to $H$ if and only if $H$ is finitely generated.

\begin{theo}
\label{res: intro - main theorem}
	Let $G$ be a finitely generated group and $H$ a subgroup of $G$.
	Assume that $G$ is finitely presented relative to $H$.
	If $G$ is sparsely asymptotically tree-graded with respect to $H$ then $G$ is hyperbolic relative to $H$.
\end{theo}

Let us briefly sketch the proof.
If $G$ is finitely generated and finitely presented relative to $H$, then $H$ must be finitely generated \cite{Osin:2006cf}.
So without loss of generality, we can choose a finite generating set of $G$ that contains a generating set of $H$.
In particular, the corresponding Cayley graph $Y$ of $H$ embeds into the corresponding Cayley graph $X$ of $G$.
In addition, we can assume that $H$ is infinite.
Otherwise the asymptotic cone of $G$ would just be an $\R$-tree and the result would follow from \autoref{res: intro - main theorem for hyperbolic groups}.
We construct from this data a cone-off space $\dot X$ obtained as follows.
For every $g \in G/H$ we attach to the space $X$ an horocone $Z(gY)$ along $gY$.
This horocone is the product $Z(gY) = gY\times \R_+$ endowed with a metric modeled on an horoball of the hyperbolic plane $\H$.
Similar constructions have been considered by many people (for instance \cite{Bowditch:2012ga, DelGro08,Groves:2008ip,Coulon:il}).
In fact, if $G$ is hyperbolic relative to $H$ then the canonical action of $G$ on the cone-off $\dot X$ satisfies the axioms of \autoref{def: intro - relatively hyperbolic}.
Therefore $\dot X$ is a natural candidate to prove that $G$ is relatively hyperbolic under our assumptions.
In particular we need to show that $\dot X$ is hyperbolic which is where we use the previously stated Cartan-Hadamard theorem for hyperbolic spaces. As explained before it combines two ingredients, a topological one and a metric one.

\paragraph{}
Since $G$ is finitely presented relative to $H$ every loop in $X$ (and thus $\dot X$) can be written as a product of loops each of which is homotopic either to a loop given by a relation of our presentation or a loop contained in $gY$ for some $g \in G$.
Since the presentation is finite, there are only finitely many loop in the first class (up to conjugacy).
On the other hand, any loop of the second type can be pushed up in the horocone in such a way that its diameter becomes arbitrarily small.
Thus there exists $\epsilon>0$ such that $\dot X$ is $\epsilon$-simply-connected (see \autoref{res: cone-off simply-connected}).

\paragraph{}
By construction, every horocone of $\dot X$ is $2\boldsymbol \delta$-hyperbolic, where $\boldsymbol \delta$ is the hyperbolicity constant of the hyperbolic plane $\H$.
However, attached together, the cone-off space $\dot X$ might fail to be hyperbolic (even locally hyperbolic).
Here we use the asymptotic properties of $G$.
By assumption an asymptotic cone $X_\infty=\ascong Gd$ of $G$ is tree-graded with respect to a collection of pieces $\mathcal Y_\infty$ which correspond exactly to the limits of translates of $Y$.
In particular if we perform the cone-off construction by attaching on $X_\infty$ horocones $Z(Y_\infty)$ with base $Y_\infty \in \mathcal Y_\infty$, we obtain a $2\boldsymbol \delta$-hyperbolic space.
Indeed, if we attach together two hyperbolic spaces which have a single common point, the resulting space is still hyperbolic with the same hyperbolicity constant.
This proves that a cone-off over the asymptotic cone of $G$ is hyperbolic.
However the two operations -- building the cone-off and taking the asymptotic cone -- can be interchanged (in a precise way, see \autoref{res: cone-off vs ultra-limit}).
It follows that a limit of cone-off spaces $(\dot X_n)$ built  over the Cayley graph $X$ of $G$ is hyperbolic, therefore one of them is locally hyperbolic (with constants as good as desired).
Using \autoref{res: intro - cartan hadamard}, we get that there exists $n \in \N$ such that $\dot X_n$ is globally hyperbolic (see \autoref{res: cone-off curvature general case}).
Then we prove that it also satisfies the other assumptions of \autoref{def: intro - relatively hyperbolic}.

\paragraph{Outline of the paper.} In \autoref{sec: metric spaces} we recall the properties about metric spaces that will be useful later.
In particular we make precise the definitions  that have been sketch in the introduction (hyperbolic spaces, ultra-limits, tree-graded spaces, etc).
At the end of the section we reformulate our main theorem in terms of an action on a metric space (see \autoref{res: main theorem geometric form}).
\autoref{sec: cone-off} is devoted to proof of a slightly weaker form of \autoref{res: main theorem geometric form} (see \autoref{res: main theorem weak geometric form}).
We use as explained in the introduction the cone-off construction with horocones.
In \autoref{sec: spaces with a tree graded asymptotic cone}, we explain how the strong statement (\autoref{res: main theorem geometric form}) can actually be deduced from the weaker one (\autoref{res: main theorem weak geometric form}).
It uses a careful study of tree-graded ultra limits.
The last section provides examples, in particular \autoref{res: intro - main theorem} and comments about the main result.

\paragraph{Acknowledgement.} We would like to thank Mark Sapir for his  helpful conversations on this topic as well as introducing the question to the authors.


%
%

\section{Generalities about metric spaces.}
\label{sec: metric spaces}

\paragraph{}
In this section we collect some definitions and properties about metric spaces.
Let $X$ be a metric space.
Given two points $x$, $x'$ in $X$, we write $\dist[X] x{x'}$ or simply $\dist x{x'}$ for the distance between them.
The open ball of radius $r$ and center $x$ is denoted by $\ball xr$.
The space $X$ is \emph{proper} if every bounded closed subset of $X$ is compact.
Let $I$ be an interval of $\R$.
The length of a rectifiable path $\gamma : I \rightarrow X$ is denoted by $L(\gamma)$.

\subsection{Ultra-limit of metric spaces and asymptotic cones}
\label{sec: ultra-limits}

\paragraph{}
A \emph{non-principal utltra-filter} is a finite additive map $\omega : \mathcal P(\N) \rightarrow \{0,1\}$ which vanishes on every finite subset of $\N$ and such that $\omega(\N) = 1$.
A property $P_n$ depending on an integer $n$ is said to be true \oas if $\omega(\set{n \in \N}{P_n \text{ is true}}) =1$.
An sequence of real numbers $(u_n)$ is \oeb if there exists a number $M$ such that $|u_n| \leq M$ \oas.
Given a real number $l$ we say that the \emph{$\omega$-limit} of $(u_n)$ is $l$ and write $\limo u_n = l$ if for every $\epsilon > 0$, $|u_n - l| \leq \epsilon$ \oas.
In particular any \oeb sequence of real numbers admits an $\omega$-limit \cite{Bou71}.

\paragraph{}
Let $(X_n,e_n)$ be a sequence of pointed metric spaces.
We define the partial product $\Pi_\omega X_n$  by
\begin{equation*}
	\Pi_\omega X_n  = \set{(x_n) \in \Pi_{n\in \N}X_n}{\dist{x_n}{e_n} \text{ is \oeb}}.
\end{equation*}
This set is endowed with a pseudo-metric given by the following formula.
Given two elements $x=(x_n)$ and  $x'=(x'_n)$ in $\Pi_\omega X_n$, $\dist x{x'} = \limo \dist{x_n}{x'_n}$.

\begin{defi}
\label{def: ultra-limit metric space}
	Let $(X_n,e_n)$ be a sequence of pointed metric spaces and $\omega$ a non-principal ultra-filter.
	The \emph{$\omega$-limit} of $(X_n, e_n)$ denoted by $\limo (X_n, e_n)$, or simply $\limo X_n$, is the quotient of $\Pi_\omega X_n$ by the equivalence relation which identifies two points at distance zero.
	The pseudo-distance on $\Pi_\omega X_n$ induces a distance on $\limo X_n$.
\end{defi}

\notas With the same notations as in the previous definition.
\begin{itemize}
	\item Let $(x_n)$ be an element of $\Pi_\omega X_n$.
	The image of $(x_n)$ in $\limo X_n$ is denoted $\limo x_n$.
	\item For every $n \in \N$, let $Y_n$ be a subset of $X_n$.
	We define the subset $\limo Y_n$ of $\limo X_n$ by
	\begin{equation*}
		\limo Y_n = \set{\limo y_n}{(y_n) \in \Pi_\omega X_n \text{ and } y_n \in Y_n \text{ \oas}}.
	\end{equation*}
\end{itemize}

The asymptotic cone of a metric space is a particular example of ultra-limit.
It is defined as follows.
\begin{defi}
\label{def: asymptotic cone}
	Let $\omega$ be a non-principal ultra-filter.
	Let $X$ be a metric space.
	Let $e=(e_n)$ be a sequence of points of $X$ and $d=(d_n)$ a sequence a real numbers diverging to infinity.
	The \emph{asymptotic cone of $X$ with respect to $e$, $d$ and $\omega$} and denoted by $\ascon Xd$ is the $\omega$-limit $\limo ((1/d_n)X,e_n)$ where $(1/d_n)X$ stands for the space $X$ whose metric has been rescaled by $d_n$.
\end{defi}

\subsection{Tree-graded spaces.}

\begin{defi}[Due to C.~Drutu and M.~Sapir] {\rm \cite{DruSap05}} \quad
	Let $X$ be a complete geodesic metric space and $\mathcal Y$ a collection of closed geodesic subsets (called \emph{pieces}).
	Suppose that the following two properties are satisfied.
	\begin{indexedenu}[T]
	    \item\label{T1} Every two different pieces have at most one common point.
	    \item\label{T2} Every simple closed curve in $X$ is contained in one piece.
	\end{indexedenu}
	Then we say that the space $X$ is \emph{tree-graded with respect to $\mathcal Y$}.
\end{defi}

\begin{defi}
	Let $X$ be tree-graded with respect to $\mathcal Y$ a collection of closed subsets of $X$.
	Suppose that $\gamma:S^1\to X$ is a continuous map which intersects a piece $Y$  from $\mathcal Y$.
	Then $S^1\backslash \{\gamma^{-1}(M)\}$ is the union of disjoint open intervals or empty.
	The \emph{$Y$-transition points} of $\gamma$ are the endpoints of the connected components of $S^1\backslash \{\gamma^{-1}(Y)\}$.
	A point of $S^1$ is a \emph{transition point} of $\gamma$ if it is an $Y$-transition point for some piece $Y$.
\end{defi}

\begin{lemm}
\label{M-transitions}
	Suppose that $X$ is tree-graded by closed pieces.
	If $\gamma: S^1\to X$ is not contained in a single piece, then there exists distinct transition points $t_1,t_2\in S^1$ such that $\gamma(t_1) = \gamma(t_2)$.
\end{lemm}

\begin{proof}
	Without loss of generality we can assume that $\gamma$ is not a constant map.
	Since $X$ is tree graded  $\gamma^{-1}(Y)$ is non-degenerate for some piece $Y$.
	Fix such a piece $Y$.
	Suppose that $x,y$ are the two endpoints of a connected component $[x,y]$ of $S^1\backslash \{\gamma^{-1}(Y)\}$.
	($S^1\backslash \{\gamma^{-1}(Y)\}$ is non-empty since $\gamma$ is not contained in single piece.)
	Since $\gamma^{-1}(Y)$ is non-degenerate, $x\neq y$.
	Let $(x,y) = [x,y]\backslash \{x,y\}$, $[x,y) = [x,y]\backslash \{y\}$, and $(x,y] = [x,y]\backslash \{x\}$. Then $\gamma(x),\gamma(y) \in Y$ and $\gamma\bigl((x,y)\bigr) \cap Y = \emptyset$.
	Following proof of \cite[Corollary 2.9]{DruSap05} of $\gamma\bigl([x,y)\bigr)$ projects onto $\gamma(x)$ and $\gamma\bigl((x,y]\bigr)$ projects onto $\gamma(y)$.
	Since projection onto a piece is a well-defined map, $\gamma(x) = \gamma(y)$.
\end{proof}

\begin{defi}
	Fix a non-principal ultra-filter $\omega$.
	For every $n \in \N$, let  $(X_n,e_n)$ be a pointed geodesic metric space and $\mathcal Y_n $ a family of subspaces of $X_n$.
	We will then say that $(X_n,e_n)$ is \emph{sparsely tree-graded} with respect to $(\mathcal Y_n)$, if the following assertions hold.
	\begin{indexedenu}[T^\omega]
	    \item\label{Tw1}
	    For every $(Y_n), (Y_n')\in \Pi_{n \in \N}\mathcal Y_n$, if $Y_{n} \neq Y_{n}'$ \oas, then the subsets  $\limo Y_{n}$ and $\limo Y_{n}'$ of $\limo (X_n,e_n)$ have at most one common point.
	   \item\label{Tw2}
	   Every simple closed curve in $\limo (X_n,e_n)$ is contained in some $\limo Y_n$.
	\end{indexedenu}

\end{defi}
Note that being sparsely tree-graded depends on an ultrafilter and is slightly stronger than $\limo (X_n, e_n)$ being tree-graded with respect to $\mathcal Y = \set{\limo Y_n}{(Y_n) \in \Pi_{n \in \N}\mathcal Y_n}$.
Let $X$ be a metric space and $\mathcal Y$ be a collection of subspaces of $X$.
We will say that $X$ is \emph{sparsely asymptotically tree-graded with respect to $\mathcal Y$}, if there exists a non-principal ultra-filter $\omega$, a sequence $d = (d_n)$ of numbers diverging to infinity and a sequence $e=(e_n)$ of base points of $X$ such that  such that $(X_n,e_n)$ is sparsely tree-graded with respect to $(\mathcal Y_n)$, where for every $n \in \N$,
\begin{itemize}
	\item $X_n = (1/d_n)X$ is the space $X$ rescaled by $d_n$.
	\item $\mathcal Y_n = \set{(1/d_n)Y}{Y \in \mathcal Y}$ is the image in $X_n$ of the collection $\mathcal Y$.
\end{itemize}
When necessary for clarity, we will say that $X$ is $(\omega, d,e)$-sparsely asymptotically tree-graded.
If a metric space is $(\omega, d,e)$-sparsely asymptotically tree-graded for every triple $(\omega, d,e)$, then $X$ is \emph{asymptotically tree-graded} in the terminology of C.~Drutu and M.~Sapir \cite{DruSap05}.
Let $G$ be a finitely generated group and $\{H_1, \dots ,H_m\}$ be a finite collection of subgroups of $G$.
By abuse of notation we will say that the group $G$ is \emph{sparsely asymptotically tree-graded with respect to $\{H_1, \dots ,H_m\}$} if $G$ (as a metric space) is sparsely asymptotically tree-graded with respect to the set $\mathcal H$ of all $H_i$-cosets, i.e.
\begin{equation*}
	\mathcal H = \bigcup_{i =1}^m \set{gH_i}{g \in G/H_i}.
\end{equation*}

\begin{lemm}
\label{res: tree-graded ultra-limit change of base point}
Let $G$ be a group.
For every $n \in \N$, we consider a metric space $X_n$ endowed with an action by isometries of $G$ and a $G$-invariant collection $\mathcal Y_n$  of subspaces of $X_n$.
We assume that the diameter of $X_n/G$ is uniformly bounded.
Let $(e_n),(e'_n) \in \Pi_{n\in \N} X_n$ be two sequences of base points.
The spaces $\limo (X_n, e_n)$ and $\limo(X_n,e'_n)$ are isometric.
Moreover $(X_n,e_n)$ is sparsely tree-graded with respect to $(\mathcal Y_n)$ if and only if $(X_n,e'_n)$ is sparsely tree-graded with respect to $(\mathcal Y_n)$.
\end{lemm}

\rem In this paper we are mainly interested in asymptotic cones of groups.
Thus in our examples the diameter of $X_n/G$ will always be uniformly bounded.
Therefore beging sparsely tree-graded is independent of the choice of the base points.

\begin{proof}
	The diameter of the sequence $(X_n/G)$ is bounded.
	Therefore there exists a sequence $(g_n)$ of elements of $G$ such that $\dist{g_ne_n}{e'_n}$ is uniformly bounded.
	In particular for every sequence $(x_n) \in \Pi_{n\in \N}X_n$, if $(\dist{x_n}{e_n})$ is \oeb, so is $(\dist{g_nx_n}{e'_n})$.
	This allows us to define a map $\phi : \limo (X_n,e_n) \rightarrow \limo(X_n,e'_n)$ as follows.
	\begin{equation*}
		\begin{array}{lccc}
			\phi : & \limo (X_n, e_n)	& \rightarrow	& \limo (X_n, e'_n) \\
				& \limo x_n 		& \rightarrow	& \limo g_nx_n
		\end{array}
	\end{equation*}
	Since $G$ acts by isometries on $X_n$, the map $\phi$ is an isometry.
	We claim that it preserves the tree-graded structure.
	Assume that $(X_n,e'_n)$ is sparsely tree-graded with respect to $(\mathcal Y_n)$.
	\begin{enumerate}
		\item For every $n \in \N$, we choose $Y_n, Y'_n \in \mathcal Y_n$ such that $Y_n \neq Y'_n$ \oas.
		The map $\phi$ sends $\limo Y_n \cap \limo Y'_n$ to a subset of $\limo g_nY_n \cap \limo g_nY'_n$.
		However for every $n \in \N$, $\mathcal Y_n$ is $G$-invariant, thus $g_nY_n$ and $g_nY'_n$ are two elements of $\mathcal Y_n$.
		Moreover $g_nY_n \neq g_nY'_n$ \oas.
		Since $(X_n,e'_n)$ is sparsely tree-graded with respect to $(\mathcal Y_n)$, $\limo g_nY_n$ and $\limo g_nY'_n$ have at most one common point.
		It follows that $\phi(\limo Y_n \cap \limo Y'_n)$ contains at most one point.
		Recall that $\phi$ is a bijection, thus the same holds for $\limo Y_n \cap \limo Y'_n$.
		\item Let $\gamma$ be a simple closed loop of $\limo(X_n, e_n)$.
		Since $\phi$ is an isometry, $\phi\circ\gamma$ is a simple closed loop of $\limo(X_n, e'_n)$.
		Consequently, for every $n \in \N$, there exists an element $Y_n \in \mathcal Y_n$ such that $\phi \circ \gamma$ is contained in $\limo Y_n$.
		By definition of $\phi$, $\gamma$ is contained in $\limo g_n^{-1}Y_n$.
		Recall that $\mathcal Y_n$ is $G$-invariant, thus for every $n \in \N$, $g_n^{-1}Y_n$ belongs to $\mathcal Y_n$.
		\end{enumerate}
	These two points show that $(X_n, e_n)$ is sparsely tree-graded with respect to $(\mathcal Y_n)$.
\end{proof}

\subsection{Hyperbolic spaces}
\label{sec : hyperbolic spaces}

\paragraph{The four point inequality.}
Let $x$, $y$ and $z$ be three points of $X$.
Their \emph{Gromov product} is defined by the following formula.
\begin{equation*}
	\gro xyz = \frac 12 \left(\fantomB \dist xz + \dist yz - \dist xy \right)
\end{equation*}

\begin{defi}
\label{def: hyperbolic space}
	Let $\delta \geq 0$.
	The metric space $X$ is \emph{$\delta$-hyperbolic} if for every $x,y,z,t \in X$,
	\begin{equation}
	\label{eqn: four point condition}
		\gro xzt \geq \min \left\{ \fantomB \gro xyt, \gro yzt\right\}-\delta.
	\end{equation}
\end{defi}

Note that in this definition we did not assume $X$ to be proper or geodesic.
Let $\sigma >0$.
The space $X$ is \emph{$\sigma$-locally $\delta$-hyperbolic} if every ball of radius $\sigma$ is $\delta$-hyperbolic.
When the space $X$ is ``almost'' simply-connected the hyperbolicity actually follows from a local four point condition.

\begin{defi}
\label{def: epsilon simply-connected}
	Let $\epsilon >0$.
	The space $X$ is \emph{$\epsilon$-simply connected} if any based loop is homotopic to a product of loops each of which is freely homotopic to a loop of diameter at most $\epsilon$.
\end{defi}

\begin{theo}[Cartan-Hadamard Theorem] {\rm \cite[Theorem A.1]{Coulon:2013tx}} \quad
\label{res: cartan hadamard}
Let $\delta \geq 0$ and  $\sigma > 10^7\delta$.
Let $X$ be a length space.
If $X$ is $\sigma$-locally $\delta$-hyperbolic and $10^{-5}\sigma$-simply connected then $X$ is (globally) $300\delta$-hyperbolic.
\end{theo}

For the rest of \autoref{sec : hyperbolic spaces} we will assume that $X$ is a $\delta$-hyperbolic length space.

\begin{lemm}
\label{res: four point property}
	 Let $x_1$, $x_2$, $x'_1$ and $x'_2$ be four points of $X$.
	 Let $y \in X$.
	 If
	 \begin{equation*}
	 	\dist {x_1}y > \dist {x_1}{x_2} + \gro {x_1}{x'_1}y + 2\delta \text{ and } \dist {x'_1}y > \dist {x'_1}{x'_2} + \gro {x_1}{x'_1}y + 2\delta,
	 \end{equation*}
 then $\gro {x_2}{x'_2}y \leq \gro {x_1}{x'_1}y  + 2\delta$.
\end{lemm}

\begin{proof}
	Applying twice the four point inequality~(\ref{eqn: four point condition}) we get
	\begin{equation}
	\label{eqn: four point property}
		\min \left\{ \gro{x_1}{x_2}y, \gro {x_2}{x'_2}y, \gro{x'_1}{x'_2}y \right\} \leq \gro {x_1}{x'_1}y  + 2\delta
	\end{equation}
	However the triangle inequality gives $\gro{x_1}{x_2}y \geq \dist {x_1}y - \dist{x_1}{x_2}$.
	According to our first assumption on $y$ the minimum in~(\ref{eqn: four point property}) cannot be achieved by $\gro{x_1}{x_2}y$.
	Similarly it cannot achieved by $\gro{x'_1}{x'_2}y$.
	Thus $\gro{x_2}{x'_2}y \leq \gro {x_1}{x'_1}y  + 2\delta$.
\end{proof}

\paragraph{Quasi-geodesics.}
Let $I$ be an interval of $\R$.
Let $\gamma : I \rightarrow X$ be a rectifiable path of $X$ parametrized by arclength.
It is an \emph{$L$-local $(1,l)$-quasi-geodesic} if for every $s,t \in I$ such that $\dist st \leq L$ we have $\dist st \leq \dist{\gamma(s)}{\gamma(t)} + l$.
It is a \emph{$(1,l)$-quasi-geodesic}, if for every $L \geq 0$, $\gamma$ is an $L$-locally $(1,l)$-quasi-geodesic.

\begin{lemm}{\rm \cite[Proposition 2.4]{Coulon:2013tx}} \quad
\label{res: quasi-geodesic - quasi-convex}
	Let $\gamma : I \rightarrow X$ be a $(1,l)$-quasi-geodesic joining two points $y$ and $y'$.
	For every $x \in X$, we have $d(x,\gamma) \leq \gro y{y'}x + l + 3 \delta$.
\end{lemm}

One important feature of hyperbolic spaces is the stability of quasi-geodesics recalled below.

\begin{prop}{\rm \cite[Corollary 2.6]{Coulon:2013tx}} \quad
\label{res: stability quasi-geodesics}
	Let $l \geq 0$.
	There exists $L >0$ depending only on $l$ and $\delta$ with the following property.
	Let $\gamma$ and $\gamma'$ be two $L$-local $(1,l)$-quasi-geodesics of $X$.
	If they have the same endpoints then the Haussdorf distance between them is at most $2l + 5\delta$.
\end{prop}

\paragraph{The boundary at infinity.}
Let $e$ be a base point of $X$.
A sequence $(y_n)$ of points of $X$ \emph{converges to infinity} if $\gro {y_n}{y_m}e$ tends to infinity as $n$ and $m$ approach infinity.
The set $\mathcal S$ of such sequences is endowed with a binary relation defined as follows.
Two sequences $(y_n)$ and $(z_n)$ are related if
\begin{displaymath}
	\lim_{n \rightarrow + \infty} \gro {y_n}{z_n}e = + \infty.
\end{displaymath}
If follows from (\ref{eqn: four point condition}) that this relation is actually an equivalence relation.
The boundary at infinity of $X$ denoted by $\partial X$ is the quotient of $\mathcal S$ by this relation.
If the sequence $(y_n)$ is an element in the class of $\xi \in \partial X$ we say that $(y_n)$  \emph{converges} to $\xi$ and  write
\begin{displaymath}
	\lim_{n \rightarrow + \infty} y_n = \xi.
\end{displaymath}
Note that the definition of $\partial X$ does not depend on the base point $e$.
If $Y$ is a subset of $X$ we denote by $\partial Y$ the set of elements of $\partial X$ which are the limit of a sequence of points of $Y$.

\paragraph{Isometries.}
The isometries of $X$ can be sort into three categories \cite[Chapitre 9, Th\'eor\`eme 2.1]{CooDelPap90}.
An isometry $g$ of $X$ is
\begin{enumerate}
	\item \emph{elliptic}, if one (and hence all) orbit of $g$ is bounded.
	\item \emph{parabolic}, if one (and hence all) orbit of $g$ admits a unique accumulation point in $\partial X$.
	\item \emph{hyperbolic}, if one (and hence all) orbit of $g$ admits exactly two accumulation points in $\partial X$.
\end{enumerate}
In order to measure the action of an isometry $g$ on $X$ we define the \emph{translation length} $\len g$ and the \emph{stable translation length} $\len[stable] g$.
\begin{equation*}
	\len g = \inf_{x \in X} \dist {gx}x, \quad \len[stable] g = \lim_{n \rightarrow + \infty} \frac 1n \dist{g^nx}x.
\end{equation*}
An isometry of $X$ is hyperbolic if and only if its stable translation length is positive \cite[Chapitre 10, Proposition 6.3]{CooDelPap90}.
The translation lengths are related according to the following proposition.
\begin{prop}{\rm  \cite[Chapitre 10, Proposition 6.4]{CooDelPap90}}
\label{res: translation lenghts}
	Let $g$ be an isometry of $X$.
	Its translation lengths satisfy
	\begin{equation*}
		\len[stable] g \leq \len g \leq \len[stable] g + 16 \delta.
	\end{equation*}
\end{prop}

\paragraph{Proper geodesic spaces.}
In this paragraph we assume that in addition to being $\delta$-hyperbolic, $X$ is also proper and geodesic.
Let $\rho : [0, + \infty) \rightarrow X$ be a geodesic ray.
There exists a point $\xi \in \partial X$ such that for every sequence of real numbers $(t_n)$  diverging to infinity, $\lim_{n \rightarrow + \infty} \rho(t_n) = \xi$.
In this situation we consider $\xi$ as an endpoint at infinity of $\rho$ and write $\lim_{t \rightarrow + \infty}\rho(t) = \xi$.
For every pair of distinct points $x,x' \in X \cup \partial X$ there exists an (eventually infinite) geodesic $\gamma$ joining $x$ to $x'$ \cite[Chapitre 2, Proposition 2.1]{CooDelPap90}.

\paragraph{}
Since $X$ is proper and geodesic, $\partial X$ is in one-to-one correspondence with the quotient of the set of rays starting at a given base point $e$ by the equivalence relation that identifies two rays at finite Hausdorff distance \cite[Chapitre 2, Proposition 3.1]{CooDelPap90}.
Therefore $X \cup \partial X$ inherits the topology of uniform convergence on bounded sets.
Moreover $X \cup \partial X$ is compact for this topology \cite[Chapitre 2, Proposition 3.2]{CooDelPap90}.

\paragraph{}
Let $\xi \in \partial X$ and $\rho : [0, + \infty) \rightarrow X$ be a geodesic ray such that $\lim_{t \rightarrow + \infty} \rho(t) = \xi$.
We associate to $\rho$ a function $h : X \rightarrow \R$ defined by
\begin{equation*}
	h(x) = \limsup_{t \rightarrow + \infty} \dist x{\rho(t)} - t.
\end{equation*}
Such a map is called a \emph{Buseman function about the point $\xi$}.
\begin{defi}
\label{def: horoball}
	Let $\xi \in \partial X$.
	A subset $Y$ of $X$ is a \emph{horoball centered at $\xi$} if there exists a Buseman function $h$ about $\xi$ and a constant $\alpha \geq 0$ such that
 	for every $x \in Y$, $h(x) \leq \alpha$ and for every $x \in X \setminus Y$, $h(x) \geq - \alpha$.
\end{defi}

\paragraph{Relatively hyperbolic groups.}
Let $G$ be a group acting properly on a proper geodesic hyperbolic metric space $X$.
By properly we mean that for every $x\in X$ there exists a positive number $r$ such that the set $\set{g}{g\ball xr \cap \ball xr \neq \emptyset}$ is finite.
A subgroup $H$ of $G$ is called a \emph{parabolic subgroup} if $H$ is infinite and $H$ contains no hyperbolic element.
In this case $H$ has a unique fixed point $\xi\in\partial X$, called a \emph{parabolic point}. If $\xi$ is a parabolic point, $\stab \xi$ is a maximal parabolic subgroup.

\begin{defi}{\rm \cite[Definition 3.3]{Hruska:2010iw}} \quad
\label{def: relatively hyperbolic}
	Let $G$ be a group and $\{ H_1, \dots, H_m\}$ be a collection of subgroups of $G$.
	We say that $G$ is \emph{hyperbolic relative to} $\{H_1, \dots, H_m\}$ if there exists a proper geodesic hyperbolic space $X$ and a collection $\mathcal Y$ of disjoint open horoballs satisfying the following properties.
	\begin{enumerate}
		\item $G$ acts properly by isometries on $X$ and $\mathcal Y$ is $G$-invariant.
		\item If $U$ stands for the union of the horoballs of $\mathcal Y$ then $G$ acts co-compactly on $X \setminus U$.
		\item $\{H_1, \dots, H_m\}$ is a set of representatives of the $G$-orbits of $\set{\stab Y}{Y \in \mathcal Y}$.
	\end{enumerate}
\end{defi}

It follows from this definition that for every $j \in \intvald 1m$, the conjugates of $H_j$ are maximal parabolic subgroups for the action of $G$ on $X$.

\subsection{Statement of the main theorems}

\begin{defi}
\label{def: rectifiable path connected}
Let $Y$ be subset of a metric space $X$.
We say that $Y$ is \emph{rectifiably path connected} if every two points of $Y$ can be joined by a rectifiable path contained in $Y$.
In this situation we denote by $\distV[Y]$ the induced length metric of $Y$ obtained by restricting $\distV[X]$ to $Y$.
If in addition there exists $\epsilon>0$ such that for every $y \in Y$, the natural embedding $Y \hookrightarrow X$ induces an isometry from $\ball y\epsilon$ onto its image, then we say that $Y$ is \emph{locally undistorted}.
\end{defi}

\begin{defi}
\label{def: relatively simply-connected}
	Let $\epsilon >0$.
	Let $\mathcal Y$ be a collection of subspaces of a metric space $X$.
	We say that $X$ is \emph{$\epsilon$-simply connected relative to $\mathcal Y$} if any based loop is homotopic to a product of loops $\gamma_1 \cdot \gamma_2 \cdots \gamma_m$ such that for every $m \in \intvald 1m$, $\gamma_i$ is freely homotopic to a loop which has either diameter bounded above by $\epsilon$ or is contained in one of the subsets of $\mathcal Y$.
\end{defi}

\paragraph{}\autoref{res: intro - main theorem} is a particular case of the following general result.

\begin{theo}
\label{res: main theorem geometric form}
	Let $G$ a group acting properly co-compactly by isometries on a proper length space $X$.
	Let $\mathcal Y$ be a $G$-invariant collection of closed locally undistorted subsets of $X$ such that $X$ is $\epsilon$-simply-connected relative to $\mathcal Y$ for some $\epsilon >0$  and $\mathcal Y/G$ is finite.
	We identify $\mathcal Y/G$ with a set of representatives of the $G$-orbits of the elements of $\mathcal Y$.
	If $X$ is sparsely asymptotically tree-graded with respect to $\mathcal Y$, then $G$ is hyperbolic relative to $\set{\stab Y}{Y \in \mathcal Y/G}$.
\end{theo}

\paragraph{}The proof of this statement has two main steps.
First we will focus on a slightly weaker form of \autoref{res: main theorem geometric form}.
This second statement (see \autoref{res: main theorem weak geometric form}) involves an additional assumption about the behavior of $\mathcal Y$ with respect to an ultra-filter as detailed below.
In \autoref{sec: spaces with a tree graded asymptotic cone}, we explain how the collection $\mathcal Y$ can be substituted  without loss of generality for an other family satisfying our additional hypothesis

\begin{defi}
\label{def: distortion omega controlled}
	Let $(X_n,e_n)$ be a sequence of pointed length spaces.
For every $n \in \N$, let $\mathcal Y_n$ be a collection of rectifiably path connected subsets of $X_n$.
Let $\omega$ be a non-principal ultra-filter.
	We say that the distortion of the sequence $(\mathcal Y_n)$ is \emph{$\omega$-controlled} if for every sequence $(Y_n) \in \Pi_{n \in \N} \mathcal Y_n$, for every $(y_n), (y'_n) \in \Pi_{n \in \N} Y_n$, we have
	\begin{displaymath}
		\limo \dist[X_n]{y_n}{y'_n} = \limo \dist[Y_n]{y_n}{y'_n}.
	\end{displaymath}
\end{defi}

\begin{theo}
\label{res: main theorem weak geometric form}
	Let $G$ be a group.
	Let $\epsilon >0$.
	Assume that for every $n \in \N$, we are given
	\begin{enumerate}
		\item a pointed proper length space $(X_n,e_n)$ on which $G$ acts properly co-compactly by isometries such that the diameter of $X_n/G$ is uniformly bounded.
		\item a $G$-invariant collection $\mathcal Y_n$ of closed unbounded locally undistorted subsets of $X_n$ such that $X_n$ is $\epsilon$-simply-connected relative to $\mathcal Y_n$ and $\mathcal Y_n/G$ is finite.
	\end{enumerate}
	Let $\omega$ be a non-principal ultra-filter.
	Assume that the distortion of $(\mathcal Y_n)$ is $\omega$-controlled and that $(X_n,e_n)$ is sparsely tree-graded with respect to $(\mathcal Y_n)$.
	Then there exists a subset $A$ of $\N$ with $\omega(A) = 1$ such that for every $n \in A$, $G$ is hyperbolic relative to $\set{\stab Y}{Y \in \mathcal Y_n/G}$ where $\mathcal Y_n/G$ is identified with a set of representatives of the $G$-orbits of the elements of $\mathcal Y_n$.
\end{theo}


%
%

\section{Cone-off over a metric space.}
\label{sec: cone-off}


\subsection{Cone modelled on a horoball.}
\label{sec: horocone}
In this section $Y$ denotes a metric space.

\begin{defi}
\label{def: horocone}
	The \emph{horocone over $Y$} denoted by $Z(Y)$ is the space $Y \times \R_+$ endowed with the metric characterized as follows.
	For every $x_1 = (y_1,r_1)$, and $x_2 = (y_2,r_2)$ in $Z(Y)$,
	\begin{equation}
	\label{eqn: def metric horocone}
		\cosh\dist {x_2}{x_1} = \cosh(r_2 - r_1) + \frac 12 e^{-(r_1+r_2)} {\dist {y_2}{y_1}}^2.
	\end{equation}
\end{defi}

\paragraph{Geometric interpretation.} The distance in the space $Z(Y)$ can be seen in the following way.
Let us denote by $\mathcal H$ the upper-half plane model of the hyperbolic plane $\H_2$.
\begin{displaymath}
	\mathcal H = \set {(u, v) \in \R^2}{v>0}, \quad ds^2 = \frac {du^2 + dv^2}{v^2}.
\end{displaymath}
Let $x_1 = (y_1,r_1)$, and $x_2 = (y_2,r_2)$ be two points of $Z(Y)$.
Let $u_0 \in \R$.
We consider comparison points in $\mathcal H$.
Let $\tilde y_1$ and $\tilde y_2$ be the points of $\mathcal H$ with respective coordinates $(u_0,1)$ and $(u_0 + \dist {y_2}{y_1},1)$.
The points $\tilde x_1$ and $\tilde x_2$ are given by $\tilde x_1 = (u_0, e^{r_1})$ and $\tilde x_2 = (u_0 +\dist {y_2}{y_1}, e^{r_2})$ (see \autoref{fig: horocone - geometric interpretation}).
Thus $\dist[\mathcal H] {\tilde y_i}{\tilde x_i} = r_i$.
The distance $\dist {x_2}{x_1}$ is exactly the distance in $\mathcal H$ between $\tilde x_1$ and $\tilde x_2$.
It is easy to check from (\ref{eqn: def metric horocone}) that the metric of $Z(Y)$ is positive and symmetric.
The triangle inequality follows from the geometric interpretation.

\begin{figure}[htbp]
\centering
	\includegraphics{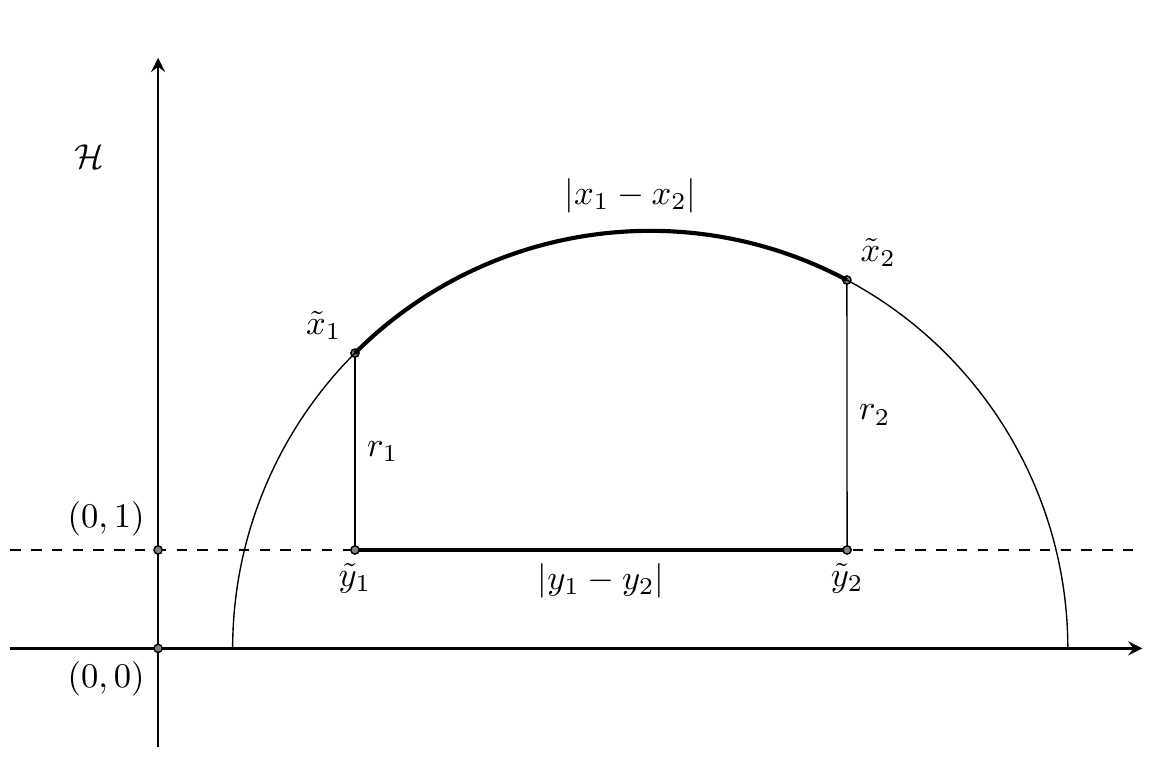}
\caption{Distance in a horocone: geometric interpretation.}
\label{fig: horocone - geometric interpretation}
\end{figure}

\paragraph{}
We denote by $\iota$ the natural map from $Y$ into $Z(Y)$ defined by $\iota(y) = (y,0)$.
On the other hand, the \emph{radial projection} $p : Z(Y) \rightarrow Y$ is the map which sends $x = (y,r)$ to $y$.
For every $y_1,y_2 \in Y$, we have
\begin{displaymath}
	\dist{\iota(y_2)}{\iota(y_1)} = \mu \left( \dist {y_2}{y_1}\right).
\end{displaymath}
where $\mu : \R_+ \rightarrow \R_+$ is the map satisfying
\begin{displaymath}
	\forall u \in \R_+, \quad \cosh(\mu(u)) = 1 + \frac 12 u^2.
\end{displaymath}
In particular, for every $u \in \R_+$, $u = 2 \sinh(\mu(u)/2)$.

\begin{lemm}
\label{res: function mu}
	The function $\mu$ is non-decreasing, concave, subadditive and 1-Lipschitz.
Moreover, for every $u \in \R_+$,
	\begin{displaymath}
		\mu(u) \geq u- \frac 1{12}u^3.
	\end{displaymath}
\end{lemm}

\rem Using the vocabulary introduced by the first author in \cite[Section 5]{Coulon:2013tx} it means that $\mu$ is an $a$-comparison map with $a = 1/12$.

\begin{proof}
	The proof is left to the reader.
\end{proof}

\begin{lemm}
\label{res: horocone - projection}
	Let $x_1=(y_1,0)$ and $x_2=(y_2,r_2)$ be two points of $Z(Y)$.
	If $\dist {x_2}{x_1} \leq d$ then $\dist {p(x_2)}{p(x_1)} \leq 2e^{d/2}\sinh(d/2)$.
\end{lemm}

\begin{proof}
	By definition of the metric of $Z(Y)$ we get that $r_2 \leq d$.
	Using again (\ref{eqn: def metric horocone}) we obtain
	\begin{equation*}
		e^{-d}\dist{p(x_2)}{p(x_1)}^2
		\leq e^{-r_2}\dist{y_2}{y_1}^2
		\leq 2\left(\cosh\dist {x_2}{x_1}-1\right) = 4 \sinh^2\left( \frac{\dist {x_2}{x_1}}2\right),
	\end{equation*}
	which provides the result.
\end{proof}

\begin{prop}
\label{res:  lifting path in cone}
	Let $x_1 = (y_1,r_1)$ and $x_2=(y_2,r_2)$ be two points of the cone $Z(Y)$.
	Let $\gamma : \intval {a_1}{a_2} \rightarrow Y$ be a rectifiable path joining $y_1$ to $y_2$.
	There exists a continuous map $r : \intval {a_1}{a_2} \rightarrow \R_+$ satisfying the following properties.
	\begin{enumerate}
		\item \label{enu: lifting path in cone - radius}
		For every $t \in \intval {a_1}{a_2}$, $r(t) \geq \min\{r_1,r_2\}$.
		\item \label{enu: lifting path in cone - length}
		Let $\nu : \intval {a_1}{a_2} \rightarrow Z(Y)$ be the path of $Z(Y)$ defined by $\nu(t) = (\gamma(t), r(t))$.
		It is a rectifiable path joining $x_1$ and $x_2$.
		Morevover for every $s,t \in \intval {a_1}{a_2}$ the lengths of $\gamma$ and $\nu$ restricted to $\intval st$ are related by
		\begin{equation*}
			\cosh\left(L\left(\restriction{\nu}{\intval st}\right)\right)
			\leq \cosh\left( r(s) - r(t)\right) + \frac 12e^{-(r(s)+r(t))}L\left(\restriction\gamma{\intval st}\right)^2.
		\end{equation*}
		\item  \label{enu: lifting path in cone - quasi-geodesic}
		Let $l \geq 0$.
		If $\gamma$ is a $(1,l)$-quasi-geodesic of $Y$ then $\nu$ is a $(1,l)$-quasi-geodesic of $Z(Y)$.
	\end{enumerate}
\end{prop}

\begin{proof}
	Without loss of generality we can assume that $\gamma$ is parametrized by arclength.
	We define several comparison points in the upper half plane model of the hyperbolic space (see \autoref{fig: horocone - lifting path}).
	Given $i \in \{1,2\}$ we denote by $\tilde y_i$ and $\tilde x_i$ the points of $\mathcal H$ with respective coordinates $(a_i,1)$ and $(a_i,e^{r_i})$ so that $\dist{\tilde x_i}{\tilde y_i}=r_i$.
	There exists a map $r : \intval {a_1}{a_2} \rightarrow \R_+$ such that the path $\tilde \nu : \intval {a_1}{a_2} \rightarrow \mathcal H$ defined by $\tilde \nu(t) = (t,e^{r(t)})$ is the geodesic between $\tilde x_1$ and $\tilde x_2$.
	Since horoballs in $\mathcal H$ are convex for every $t \in \intval {a_1}{a_2}$, we have $r(t) \geq \min\{r,r'\}$.
	As stated in the lemma, we define that path $\nu$ by $\nu(t) = (\gamma(t),r(t))$.
	It joins $x_1$ and $x_2$.
	
	\begin{figure}[htbp]
	\centering
		\includegraphics{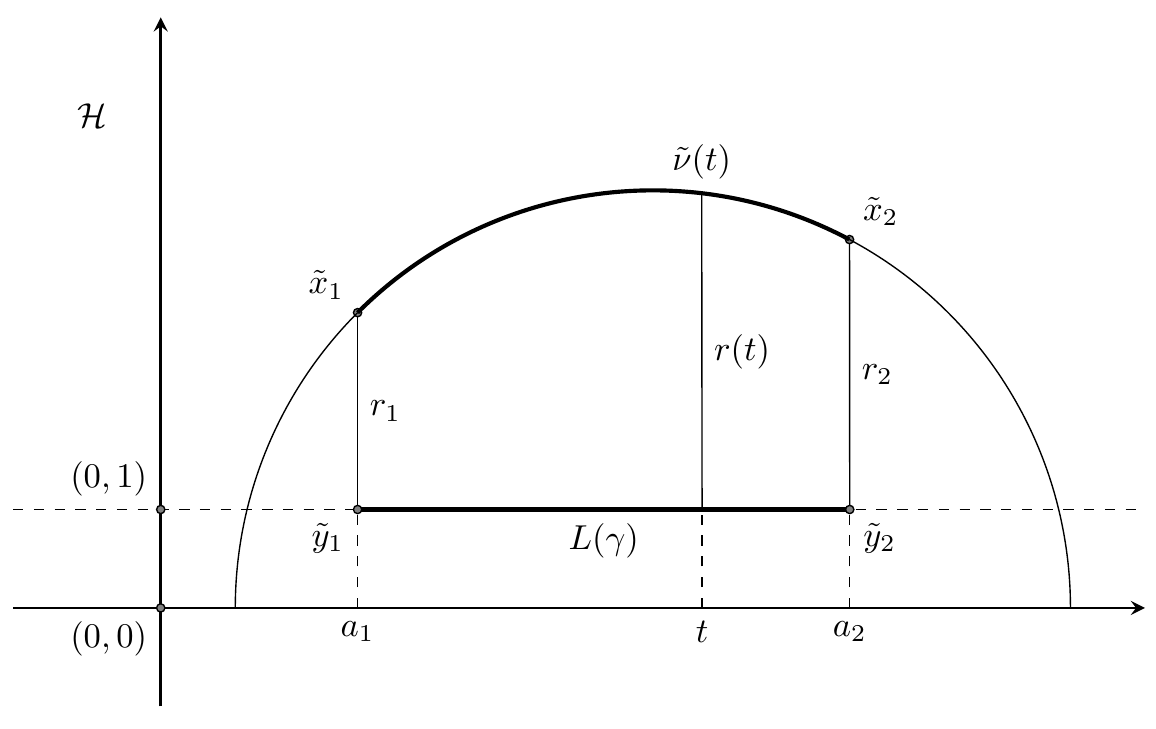}
	\caption{Lifting a path in a horocone.}
	\label{fig: horocone - lifting path}
	\end{figure}
	
	\paragraph{}We now focus on the length of $\nu$.
	Let $s,t \in \intval {a_1}{a_2}$.
	Since $\gamma$ is parametrized by arclength, $\dist{\gamma(s)}{\gamma(t)} \leq \dist st$.
	Using our geometric interpretation of the distance in $Z(Y)$ we get that $\dist{\nu(s)}{\nu(t)} \leq \dist{\tilde \nu(s)}{\tilde \nu(t)}$.
	This holds for every $s,t \in \intval {a_1}{a_2}$.
	Since $\tilde \nu$ is a geodesic of $\mathcal H$ the path $\nu$ is rectifiable.
	Moreover for every $s,t \in \intval {a_1}{a_2}$, the length of $\nu$ restricted to $\intval st$ is at most $\dist{\tilde \nu(s)}{\tilde \nu(t)}$.
	Consequently
	\begin{equation*}
		\cosh\left(L\left(\restriction{\nu}{\intval st}\right) \right)
		\leq \cosh\left(\dist{\tilde \nu(s)}{\tilde \nu(t)}\right)
		=  \cosh(r(s)-r(t)) + \frac 12 e^{-(r(s)+r(t))}L\left(\restriction\gamma{\intval st}\right)^2.
	\end{equation*}
	It only remains to prove Point~\ref{enu: lifting path in cone - quasi-geodesic}.
	Assume that $\gamma$ is a $(1,l)$-quasi-geodesic of $Y$.
	Let $s,t \in \intval {a_1}{a_2}$
	The previous inequality becomes
	\begin{equation*}
		\cosh\left(L\left(\restriction{\nu}{\intval st}\right) \right)
		\leq \cosh(r(s)-r(t)) + \frac 12 e^{-(r(s)+r(t))}\left(\dist{\gamma(s)}{\gamma(t)} + l\right)^2.
	\end{equation*}
	However for every $\alpha \in  [1, + \infty)$, for ever $\beta \in \intval 01$ the function from $\R_+$ to $\R_+$ which sends $u$ to $\arccosh(\alpha + \beta u^2/2)$ is $1$-Lipschitz.
	Consequently the length of $\nu$ restricted to $\intval st$ is at most $\dist{\nu(s)}{\nu(t)} + l$.
\end{proof}

\begin{coro}
	If $Y$ is a length space, so is $Z(Y)$.
\end{coro}

\begin{proof}
	It follows directly from \autoref{res:  lifting path in cone}.
\end{proof}

\begin{lemm}
\label{res: pushing path in a cone}
	Let $\gamma : \intval {a_1}{a_2} \rightarrow Y$ be a continuous path of $Y$ joining two points $y_1$ and $y_2$ and $d$ its diameter.
	Let $r \in \R_+$.
	The path $\nu : \intval {a_1}{a_2} \rightarrow Z(Y)$ defined by $\nu(t) = (\gamma(t),r)$ joins $x_1 = (y_1,r)$ and $x_2 = (y_2,r)$.
	Its diameter is at most $e^{-r}d$.
\end{lemm}

\begin{proof}
	Let $s,t \in \intval {a_1}{a_2}$.
	By definition of the metric on $Z(Y)$ we have
	\begin{equation*}
		\cosh\dist{\nu(s)}{\nu(t)}  = 1 + \frac 12 e^{-2r}\dist{\gamma(s)}{\gamma(t)}^2 \leq 1 + \frac 12 e^{-2r}d^2.
	\end{equation*}
	It follows that $\dist{\nu(s)}{\nu(t)} \leq \mu(e^{-r}d)$.
	Since $\mu$ is $1$-Lipschitz the diameter of $\nu$ is at most $e^{-r}d$.
\end{proof}

\paragraph{}
We now focus on the curvature of the horocones.
To that end we introduce an auxiliary cone of finite radius.
It corresponds to the cones defined in \cite[Definition 4.1]{Coulon:2013tx} with a reverse parametrization on the factor $\intval 0 \rho$.

\begin{defi}	
\label{def: cone}
	Let $\rho >0$.
	The \emph{cone of radius $\rho$ over $Y$} denoted by $Z_\rho(Y)$ is the quotient of $Y \times \intval 0\rho$ by the equivalence relation which identifies all the points of the form $(y,\rho)$.
	It is endowed with a metric characterized as follows.
	For every $x_1 = (y_1,r_1)$ and $x_2 = (y_2,r_2)$ in $Z_\rho(Y)$,
	\begin{equation}
	\label{eqn: def metric cone}
		\cosh \dist {x_2}{x_1}  = \cosh(r_2 - r_1) + \sinh(\rho-r_1)\sinh(\rho-r_2)\left[ 1 - \cos\left( \min \left\{ \pi, \frac{\dist {y_2}{y_1}}{\pi \sinh \rho}\right\}\right)\right].
	\end{equation}
\end{defi}

\begin{prop}{\rm \cite[Proposition 4.6]{Coulon:2013tx}}\quad
\label{res: cone hyperbolic}
	For every $\rho >0$, the cone of radius $\rho$ over $Y$ is $2 \boldsymbol \delta$-hyperbolic, where $\boldsymbol \delta$ denotes the hyperbolicity constant of the hyperbolic plane $\H_2$.
\end{prop}

\begin{coro}
\label{res: horocone hyperbolic}
	The horocone $Z(Y)$ is $2 \boldsymbol \delta$-hyperbolic.
\end{coro}

\begin{proof}
	Let $x = (y,r)$ and $x'=(y',r')$ be two points of $Z(Y)$.
	If $\rho \geq \max\{r,r'\}$, then one can consider the points $x_\rho$ and $x'_\rho$ of $Z_\rho(Y)$ whose coordinates are respectively $(y,r)$ and $(y',r')$.
	Note that
	\begin{displaymath}
		\lim_{\rho \rightarrow + \infty} \dist{x_\rho}{x'_\rho} = \dist x{x'}.
	\end{displaymath}
	The hyperbolicity is defined by a four point metric inequality.
	For every $\rho>0$, $Z_\rho(Y)$ is $2\boldsymbol\delta$-hyperbolic.
	By taking the limit it follows that so is $Z(Y)$.
\end{proof}

The next proposition explains that $Z(Y)$ is a horoball.
In particular the boundary at infinity of $Z(Y)$ contains exactly one point.

\begin{prop}
\label{res: horocone - horoball}
	Let $x_0 = (y_0, r_0)$ be a point of $Z(Y)$.
	We denote by $\rho : \R_+ \rightarrow Z(Y)$ the geodesic ray starting at $x_0$ defined by $\rho(t) = (y_0, r_0 +t)$.
	For every point $x = (y,r)$ of $Z(Y)$, we have
	\begin{equation*}
		\lim_{t \rightarrow + \infty} \dist{x}{\rho(t)} - t  = r_0 -r.
	\end{equation*}

\end{prop}

\begin{proof}
	Let $x = (y,r)$ be a  point of $Z(Y)$.
	Let $t \in \R_+$.
	By definition of the metric on $Z(Y)$, we have
	\begin{equation*}
		\cosh  \dist x{\rho(t)}  = \cosh (r_0+t - r) + \frac 12 e^{-(r_0+t+r)}\dist y{y_0}^2.
	\end{equation*}
	It follows that
	\begin{equation*}
		\dist x{\rho(t)}  = t + (r_0 - r) + o(1), \quad t \rightarrow + \infty. \qedhere
	\end{equation*}
\end{proof}


\subsection{Definition of the cone-off and first properties.}
\label{sec: cone-off - definition}

\paragraph{}
In this section $X$ denotes a metric length space.
We consider a family $\mathcal Y$ of rectifiably path connected subspaces of $X$.
For every $Y \in \mathcal Y$ we denote by $\distV[Y]$ the length metric of $Y$ induced by the restriction to $Y$ of the metric $\distV[X]$ of $X$.
Then $Z(Y)$ stands for the horocone over $Y$ endowed with $\distV[Y]$.
It comes with natural maps $\iota : Y \hookrightarrow Z(Y)$ and $p : Z(Y) \rightarrow Y$ (see \autoref{sec: horocone}).

\begin{defi}
\label{def: cone-off}
	The \emph{cone-off over $X$ relative to $\mathcal Y$} is the space obtained by attaching for every $Y \in \mathcal Y$ the horocone $Z(Y)$ on $X$ according to the map $\iota$.
	We denote it by $\dot X(\mathcal Y)$ or simply $\dot X$.
\end{defi}

The cone-off can be seen as the quotient of the disjoint union of $X$ and the horocones $Z(Y)$, $Y \in \mathcal Y$ by the equivalence relation that identify every point $y \in Y$ with its image $\iota(y)$ in $Z(Y)$.
In order to simplify the notation we use the same letter to design a point of this disjoint union and its equivalence class in $\dot X$.
Defined in this way, the cone-off is just a set of points.
We now explain how to endow it with a length structure.

\paragraph{} Note that the cones $Z(Y)$ are not attached isometrically to $X$.
The map $\mu$ defined in \autoref{sec: horocone} control from below the distortion between the metric on the cones and the one on the base space.
Indeed for every $Y \in \mathcal Y$, for ever $y,y' \in Y$,
\begin{equation}
\label{eqn: cone-off distortion from below}
	\mu\left(\dist[X] y{y'}\right) \leq \mu\left( \dist[Y]y{y'}\right) = \dist{\iota(y)}{\iota(y')}.
\end{equation}
Using the vocabulary introduced by the first author in \cite[Section 5]{Coulon:2013tx}, it means that the collection
\begin{displaymath}
	\mathcal Z = \set{(Z(Y), Y, \iota)}{Y \in \mathcal Y}
\end{displaymath}
is a \emph{$(\mu, X)$-family}.

\paragraph{}
We endow first the disjoint union of $X$ and the horocones $Z(Y)$ with the metric induced by their respective distances.
Let $x$ and $x'$ be two points of $\dot X$.
We define $\dist[SC] x{x'}$ to be the infimum over the distances between two points of the previous disjoint union whose classes in $\dot X$ are respectively $x$ and $x'$.
This does not define a metric.
Indeed $\distV[SC]$ does not satisfy the triangle inequality.
Therefore we introduce chains of points.
A \emph{chain} between $x$ and $x'$ is a finite sequence $C = (z_0, \dots, z_m)$ of points of $\dot X$ whose first and last points are respectively $x$ and $x'$.
Its length, denoted by $l(C)$, is
\begin{displaymath}
	l(C) = \sum_{i=0}^{m-1} \dist[SC] {z_{j+1}}{z_j}.
\end{displaymath}

\begin{prop}{\rm \cite[Proposition 5.10]{Coulon:2013tx}}\quad
\label{res: distance on the cone-off}
	For every $x, x' \in \dot X$ we put
	\begin{displaymath}
		 \dist[\dot X]x{x'} =\inf \set{l(C)}{C \text{ chain between } x \text{ and } x'}.
	\end{displaymath}
	It endows $\dot X$ with a length structure.
\end{prop}

The natural embeddings $X \hookrightarrow \dot X$ and $Z(Y) \hookrightarrow \dot X$ are not isometric.
However they are by construction 1-Lipschitz.
We now detail the relation between the metrics on $X$ and $Z(Y)$ and the one on $\dot X$.

\begin{prop}{\rm \cite[Lemma 5.8]{Coulon:2013tx}} \quad
\label{res: metric cone-off and base}
	For every $x$ and $x'$ in $X$ we have
	\begin{equation*}
		\mu\left(\dist[X] x{x'} \right) \leq \dist [\dot X] x{x'} \leq \dist[X] x{x'}.
	\end{equation*}
\end{prop}

\rem Recall that $\mu$ is a continuous map, therefore the topology on $X$ induces by $\distV[\dot X]$ is the same as the one induced by $\distV[X]$.

\begin{prop}{\rm \cite[Lemma 5.7]{Coulon:2013tx}}\quad
\label{res: metric cone-off and horocone coincide}
	Let $Y\in \mathcal Y$.
	Let $x$ be a point of $Z(Y)$.
	We denote by $d$ the distance between $x$ and $Y$ measured with the distance of $Z(Y)$.
	Let $x'$ be an other point of $\dot X$.
	If $\dist[\dot X] x{x'} < d$ then $x'$ belongs to $Z(Y)$ and $\dist[\dot X] x{x'} = \dist[Z(Y)] x{x'}$.
\end{prop}

\rem It follows from the lemma that for every rectifiable path $\gamma : \intval ab \rightarrow \dot X$, if $\gamma$ is entirely contained $Z(Y)\setminus Y$ then the lengths of $\gamma$ as a path of  $Z(Y)$ or a path of $\dot X$ are the same.

\paragraph{}
We denote by $p: \dot X \rightarrow X$ the map whose restriction to $X$ is the identity and the one to $Z(Y)$ is the radial projection onto $Y$ defined in \autoref{sec: horocone}.

\begin{lemm}
\label{res: cone-off - bounded projection}
	If $S$ is a bounded subset of $\dot X$ then $p(S)$ is a bounded subset of $X$.
\end{lemm}

\begin{proof}
	Let $S$ be a bounded subset of $\dot X$.
	There exist $x \in X$ and $d \in \R_+$ such that $S$ is contained in the ball $B(x,d)$ of $\dot X$.
	Let $x'$ be a point of $S$.
	We distinguish two cases.
	Assume first that $x'$ belongs to $X$.
	Then by \autoref{res: metric cone-off and base}, $\mu(\dist[X] x{x'}) \leq \dist[\dot X]x{x'}$.
	It follows that $\dist[X] x{p(x')} < 2\sinh (d/2)$.
	Assume now that there exists $Y \in \mathcal Y$ such that $x'$ belongs to $Z(Y)\setminus Y$.
	By definition of the metric on $\dot X$, there exists a point $z \in Y$ such that $\dist[\dot X] xz + \dist[Z(Y)] z{x'} < d$.
	In particular $\dist[\dot X] xz < d$ and $\dist[Z(Y)] z{x'} < d$.
	It follows as previously that $\dist[X] xz < 2 \sinh(d/2)$.
	Moreover by \autoref{res: horocone - projection},
	\begin{equation*}
		\dist[X]{p(z)}{p(x')} \leq \dist[Y]{p(z)}{p(x')} < 2e^{d/2}\sinh\left(d/2\right).
	\end{equation*}
	By the triangle inequality,
	\begin{equation*}
		\dist[X]x{p(x')} \leq \dist[X] xz + \dist[X]{p(z)}{p(x')} < 2(e^{d/2}+1) \sinh(d/2).
	\end{equation*}
	Consequently, $p(S)$ is contained in the ball of $X$ of center $x$ and radius $2(e^{d/2}+1) \sinh(d/2)$.
\end{proof}

We conclude this section by some topological properties of the cone-off.

\begin{lemm}
\label{res: deformation retract}
	The base space $X$ is a deformation retract of $\dot X$.
\end{lemm}

\begin{proof}
	We consider the map $F : \dot X \times \intval 01 \rightarrow X$ defined as follows.
	\begin{itemize}
		\item for every $x$ in $X$, for every $t \in \intval 01$, $F(x,t) = x$,
		\item for every $Y \in \mathcal Y$, for every $x = (y,r)$ in $Z(Y)$, for every $t \in \intval 01$, $F(x,t)$ is the point of $Z(Y)$ defined by $F(x,t) = (y,(1-t)r)$.
	\end{itemize}
	The map $F$ is a deformation retraction of $\dot X$ onto $X$.
\end{proof}

\begin{prop}
\label{res: cone-off simply-connected}
	Let $\epsilon >0$
	Assume that $X$ is $\epsilon$-simply-connected relative to $\mathcal Y$.
	Then the cone-off $\dot X$ is $\epsilon$-simply-connected.
\end{prop}

\begin{proof}
	Let $x_0$ be a base point of $X$.
	Let $\gamma$ be a loop of $\dot X$ based at $x_0$.
	According to \autoref{res: deformation retract} $\gamma$ is homotopic relative to $x_0$ to a loop $\gamma'$ contained in $X$.
	By assumption $\gamma'$ is homotopic to a product of loops $\gamma_1 \cdot \gamma_2 \cdots \gamma_m$ such that for every $m \in \intvald 1m$, $\gamma_i$ is freely homotopic to a loop which has either diameter bounded above by $\epsilon$ or is contained in some $Y \in \mathcal Y$.
	By Lemmas~\ref{res: pushing path in a cone} and \ref{res: metric cone-off and horocone coincide} a loop contained in some $Y \in \mathcal Y$ can be pushed up in the horocone $Z(Y)$ so that its diameter is at most $\epsilon$.
\end{proof}


\subsection{Cone-off and hyperbolicity}
\label{sec: cone-off - ultra-limit}

\paragraph{}
In this section we try to understand the curvature of the cone-off over a metric space.
To that end we use a limit argument.
The proposition below corresponds to the limit case.

\begin{prop}{\rm \cite[Prop. 5.26]{Coulon:2013tx}} \quad
\label{res: cone-off curvature tree graded case}
	Let $X$ be a metric space and $\mathcal Y$ a collection of subspaces of $X$.
	If $X$ is tree-graded with respect to $\mathcal Y$ then $\dot X(\mathcal Y)$ is $2 \boldsymbol \delta$-hyperbolic.
\end{prop}

Assume now that $(X_n)$ is a sequence of metric spaces whose ultra-limit $X$ is tree-graded.
We would like to capture information about the cone-off over $X_n$ from the cone-off over $X$.
This requires to understand the behavior of the cone-off construction with respect to ultra-limit of metric spaces.
It was done by the first author for the case of cones with finite radius in \cite[Section 3.5]{Coulon:il} and generalized for arbitrary cones in \cite[Section 5.3]{Coulon:2013tx}.
We first recall the relation between cone-off and ultra-limit (see \autoref{res: cone-off vs ultra-limit}).
Then we use this property to investigate the hyperbolicity of the cone-off (see \autoref{res: cone-off curvature general case}).

\paragraph{}
Let $(X_n,e_n)$ be a sequence of pointed geodesic metric spaces.
For every $n \in \N$, we choose a family $\mathcal Y_n$ of rectifiably connected subsets of $X_n$.
We write $\dot X_n$ for the cone-off over $X_n$ relative to $\mathcal Y_n$.
Let $\omega$ a non-principal ultra-filter.
We denote by $X = \limo X_n$ the ultra-limit of $(X_n, e_n)$.
Given a sequence $(Y_n) \in \Pi_{n \in \N} \mathcal Y_n$ we define a (possibly empty) subset of $X$
\begin{displaymath}
	\limo Y_n = \set{\limo y_n}{y_n \in Y_n \text{ $\omega$-as.}}.
\end{displaymath}
We write $\Pi_\omega \mathcal Y_n$ for the set of sequences $(Y_n)$ such that $\limo Y_n$ is not empty.
We endow this set with an equivalence relation.
Given two sequences $(Y_n)$ and $(Y'_n)$ of $\Pi_\omega \mathcal Y_n$,  $(Y_n) \sim (Y'_n)$ if $Y_n = Y'_n$ \oas.
In particular they define the same limit set $\limo Y_n = \limo Y'_n$.
We denote by $\mathcal Y$ the collection of limit sets $Y = \limo Y_n$ where $(Y_n) \in \Pi_\omega \mathcal Y_n/\sim$.
Finally $\dot X$ stands for the cone-off over $X$ relative to $\mathcal Y$.
Our goal is to compare $\limo (\dot X_n, e_n)$ and $\dot X$.

\paragraph{}
As we explained in \autoref{sec: cone-off - definition} the map $\mu$ provides a way to control from below the distortion between the cones attached in the cone-off and the base space.
For our purpose we need a more accurate control of this distortion.
Recall that the distortion of the sequence $(\mathcal Y_n)$ is \emph{$\omega$-controlled} if for every sequence $(Y_n) \in \Pi_{n \in \N} \mathcal Y_n$, for every $(y_n), (y'_n) \in \Pi_{n \in \N} Y_n$, we have
\begin{displaymath}
	\limo \dist[X_n]{y_n}{y'_n} = \limo \dist[Y_n]{y_n}{y'_n}.
\end{displaymath}
This definition is slightly different from the one given by the first author in \cite{Coulon:2013tx}.
However if the distortion of $(\mathcal Y_n)$ is $\omega$-controlled (in our sense) then for every $(Y_n) \in \Pi_{n \in \N} \mathcal Y_n$, for every $(y_n), (y'_n) \in \Pi_{n \in \N} Y_n$, we have
\begin{equation*}
	\limo \mu\left(\dist[X_n]{y_n}{y'_n}\right)
	= \limo \mu\left( \dist[Y_n]{y_n}{y'_n}\right)
	= \limo \dist[Z(Y_n)]{\iota_n(y_n)}{\iota_n(y'_n)},
\end{equation*}
which corresponds exactly to Definition~5.15 of \cite{Coulon:2013tx}.

\begin{theo}
\label{res: cone-off vs ultra-limit}
	Assume that the distortion of the sequence $(\mathcal Y_n)$ is $\omega$-controlled, then the spaces $\limo (\dot X_n, e_n)$ and $\dot X$ are isometric.
\end{theo}

\begin{proof}
	To compare $\limo (\dot X_n, e_n)$ and $\dot X$ we introduce two kind of maps.
	\begin{displaymath}
		\begin{array}{lcccclccc}
			\psi:	& X			& \rightarrow	& \limo \dot X_n 	& \quad	& \psi_Y:	& Z(Y)		& \rightarrow	& \limo \dot X_n \\
					& \limo x_n	& \rightarrow	& \limo x_n		&		& 		& (\limo y_n,r)	& \rightarrow	& \limo (y_n,r)
		\end{array}
	\end{displaymath}
	The second kind of map is defined for every $Y \in \mathcal Y$.
	Note that for every $Y \in \mathcal Y$, for every $y \in Y$, $\psi_Y\circ \iota (y) = \psi(y)$.
	Therefore there exists a map $\dot \psi$ from $\dot X$ to $\limo \dot X_n$ whose restriction to $X$ (\resp $Z(Y)$) is $\psi$ (\resp $\psi_Y$).
	According to \cite[Proposition 5.16]{Coulon:2013tx} for every $t \geq 0$, the map $\dot \psi$ induces an isometry from $B(e_n, \mu(t)/2)$ onto $B(\dot\psi(e_n), \mu(t)/2)$.
	However the function $\mu$ is not bounded.
	Hence $\dot \psi$ is an isometry.
\end{proof}

\begin{theo}
\label{res: cone-off curvature general case}
	Let $\epsilon >0$.
	Assume that for every sequence of base points $(e_n) \in \Pi_{n \in \N} X_n$,  $(X_n,e_n)$ is sparsely tree-graded with respect to $(\mathcal Y_n)$.
	In addition suppose that
	\begin{enumerate}
		\item \label{enu: cone-off curvature general case - simply-connected assumption}
		for every $n \in \N$, $X_n$ is $\epsilon$-simply-connected relative to $\mathcal Y_n$.
		\item \label{enu: cone-off curvature general case - distortion assumption}
		the distortion of the sequence $(\mathcal Y_n)$ is $\omega$-controlled.
	\end{enumerate}
	Then the cone-off $\dot X_n(\mathcal Y_n)$ over $X_n$ relative to $\mathcal Y_n$ is $900\boldsymbol \delta$-hyperbolic \oas.
\end{theo}

\begin{proof}
The proof of the theorem proceeds in two steps.
First we use \autoref{res: cone-off vs ultra-limit} to compare the $\omega$-limit of $\dot X_n$ and the cone-off over $\limo X_n$.
However $\limo X_n$ is assumed to be tree-graded.
It follows that $\dot X_n$ is locally hyperbolic.
The second step consists in going from local hyperbolicity to global hyperbolicity using the Cartan-Hadamard Theorem (see \autoref{res: cartan hadamard}).
More precisely it works as follows.

\paragraph{}
	Recall that $\boldsymbol \delta$ stands for the hyperbolicity constant of the hyperbolic place $\H_2$.
	Let us fix $\rho > \max\{10^{10} \boldsymbol\delta, 10^6\epsilon\}$.
	We denote by $A$ the set of integers $n \in \N$ such that for every $x \in \dot X_n$ the ball of center $x$ and radius $\rho$ in $\dot X_n$ is $3\boldsymbol\delta$ hyperbolic.
	We claim that $\omega(A) = 1$.
	Assume on the contrary that our assertion is false.
	Then for every $n \in \N$ there exists a point $x_n$ in $\dot X_n$ such that \oas the ball of center $x_n$ and radius $\rho$ is \emph{not} $3\boldsymbol\delta$ hyperbolic.
	Note that we can choose $x_n$ in the $3\rho$-neighborhood of $X_n$.
	Indeed if this is not the case then $B(x_n, \rho)$ is entirely contained in a horocone $Z(Y_n)$ and the metrics of $\dot X_n$ and $Z(Y_n)$ coincide on this ball (see \autoref{res: metric cone-off and horocone coincide}), thus by \autoref{res: horocone hyperbolic}, $B(x_n,\rho)$ is $2\boldsymbol\delta$-hyperbolic.
	We now denote by $e_n$ a projection of $x_n$ onto $X_n$.
	
	\paragraph{}Since the ball $B(x_n,\rho)$ is not $3\boldsymbol \delta$-hyperbolic, it contains four points $u_n$, $v_n$, $w_n$ and $t_n$ such that
	\begin{displaymath}
		\gro {u_n}{w_n}{t_n} < \min \left\{ \gro {u_n}{v_n}{t_n}, \gro {v_n}{w_n}{t_n}\right\} - 3 \boldsymbol \delta.
	\end{displaymath}
	However these points remain at a distance uniformly bounded of $e_n$.
	Consequently they define four points $u$, $v$, $w$ and $t$ of $\limo (\dot X_n, e_n)$ such that
	\begin{equation}
	\label{eqn: cone-off curvature general case - gromov product limit}
		\gro uwt \leq \min \left\{ \gro uvt, \gro vwt\right\} - 3 \boldsymbol \delta.
	\end{equation}
	By assumption $X = \limo (X,e_n)$ is  tree-graded with respect to $\mathcal Y$.
	It follows from \autoref{res: cone-off curvature tree graded case} that $\dot X$ is $2\boldsymbol \delta$-hyperbolic.
	On the other hand, according to \autoref{res: cone-off vs ultra-limit} the space $\limo (\dot X_n, e_n)$ and $\dot X(\mathcal Y)$ are isometric.
	This contradicts (\ref{eqn: cone-off curvature general case - gromov product limit}) and ends the proof of our claim.
	
	\paragraph{}
	We assumed that every $X_n$ is $\epsilon$-simply-connected relative to $\mathcal Y$.
	Thus by \autoref{res: cone-off simply-connected}, $\dot X_n$ is $\epsilon$-simply-connected.
	Applying the Cartan-Hadamard theorem, for every $n \in A$, the space $\dot X_n$ is (globally) $900\boldsymbol\delta$-hyperbolic.
\end{proof}


\subsection{Group action on a hyperbolic cone-off space.}
\label{sec: group action on a cone-off space}

\paragraph{}
In this section $X$ is a proper geodesic metric space and $G$ a group acting properly co-compactly by isometries on it.
We denote by $\mathcal Y$ a $G$-invariant collection of closed rectifiably connected subsets of $X$.
The goal of this section is to study the action of $G$ on the cone-off $\dot X(\mathcal Y)$ when the later is hyperbolic.
In particular we prove that the horocones that we attached are horoballs.
Moreover the group $G$ is relatively hyperbolic and the maximal parabolic subgroups of $G$ are exactly the stabilizers of the horocones.

\paragraph{}From now on we assume that  $\mathcal Y / G$ is finite and $\dot X(\mathcal Y)$ is $\delta$-hyperbolic.
Without loss of generality we can assume that $\delta >0$.
According to the stability of quasi-geodesics (see \autoref{res: stability quasi-geodesics}) there exists $L > 100 \delta$ with the following property.
The Hausdorff distance between two $L$-local $(1, \delta)$-quasi-geodesics of $\dot X$ joining the same extremities is at most $7\delta$.
Our first proposition generalizes \autoref{res: metric cone-off and horocone coincide}.

\begin{prop}
\label{res: metric cone-off and horocone globally coincide}
	Let $Y \in \mathcal Y$.
	Let $x=(y,r)$ and $x'=(y',r')$ be two points of $Z(Y)$ with $\min\{ r,r'\} > L$.
	If $p$ is a point of $\dot X$ with $\gro x{x'}p < L - 11\delta$ then $p$ belongs to $Z(Y) \setminus Y$.
	Moreover $\dist[\dot X] x{x'} = \dist[Z(Y)]x{x'}$.
\end{prop}

\begin{proof}
	Recall that the embedding $Z(Y) \hookrightarrow \dot X$ is $1$-Lipschitz.
	Hence it is sufficient to prove that $\dist[Z(Y)] x{x'} \leq \dist[\dot X]x{x'}$.
	Let $\eta \in (0, \delta)$.
	Since $Y$ endowed with the induced metric is a length space, there exists a path $\gamma : \intval ab \rightarrow Y$ such that $L(\gamma) \leq \dist[Y]y{y'} + \eta$.
	In particular $\gamma$ is a $(1,\eta)$-quasi-geodesic of $Y$.
	By \autoref{res:  lifting path in cone} there exists a function $r : \intval ab \rightarrow \R_+$ with the following properties.
	\begin{enumerate}
		\item For every $t \in \intval ab$, $r(t) > L$.
		\item Let $\nu : \intval ab \rightarrow Z(Y)$ be the path of $Z(Y)$ defined by $\nu(t) = (\gamma(t), r(t))$.
		The length of $\nu$ is at most $\dist[Z(Y)]x{x'} + \eta$.
	\end{enumerate}
	By construction $\nu$ is entirely contained in $Z(Y)\setminus Y$.
	Thus its length as a path of $\dot X$ or $Z(Y)$ is the same.
	Without loss of generality we can assume that $\nu$ is parametrized by arclength.
	Thus $\nu$ is a $(1, \eta)$-quasi-geodesic of $Z(Y)$.
	We claim that $\nu$ is a $L$-local $(1, \eta)$-quasi-geodesic of $\dot X$.
	Let $s,t \in \intval ab$ such $\dist st \leq L$.
	We have,
	\begin{equation*}
		\dist[\dot X]{\nu(s)}{\nu(t)} \leq \dist[Z(Y)]{\nu(s)}{\nu(t)} \leq \dist st \leq L.
	\end{equation*}
	Since $r(s) > L$ the distance from $\nu(s)$ to $Y$ is larger than $L$.
	According \autoref{res: metric cone-off and horocone coincide} we get
	\begin{equation*}
		\dist st \leq \dist[Z(Y)]{\nu(s)}{\nu(t)} + \eta = \dist[\dot X]{\nu(s)}{\nu(t)} + \eta,
	\end{equation*}
	which completes the proof of our claim.
	Let us now consider a path $\nu' : \intval cd \rightarrow \dot X$ joining $x$ to $x'$ such that $L(\nu') \leq \dist[\dot X] x{x'} + \eta$.
	In particular $\nu'$ is also a $L$-local $(1,\eta)$-quasi-geodesic.
	By choice of $L$, the Hausdorff distance between $\nu$ and $\nu'$ is at most $7 \delta$.
	Thus $\nu'$ is also contained in $Z(Y) \setminus Y$.
	More precisely, the distance between any point of $\nu'$ and $Y$ is at least $L - 7\delta$.
	Let $p$ be a point of $\dot X$ such that $\gro x{x'}p < L - 11\delta$ (the Gromov product being measured with the metric of $\dot X$).
	According to \autoref{res: quasi-geodesic - quasi-convex}, $d(p,\nu') < L -7\delta$.
	Thus is belongs to $Z(Y)\setminus Y$.
	
	\paragraph{}
	Let us now prove the last assertion,
	As a path of $Z(Y)$ the length of  $\nu'$ is the same in $\dot X$ or in $Z(Y)$.
	Consequently
	\begin{equation*}
		\dist[Z(Y)] x{x'} \leq L(\nu') \leq \dist[\dot X]x{x'} + \eta.
	\end{equation*}
	This last inequality holds for every sufficiently small $\eta >0$, hence $\dist[Z(Y)] x{x'} \leq \dist[\dot X]x{x'}$.
\end{proof}

\begin{prop}
\label{res: finitely many coset in a ball}
	Let $K$ be a compact subset of $X$.
	There are only finitely many $Y \in \mathcal Y$ such that $Y \cap K \neq \emptyset$.
\end{prop}

\begin{proof}	
	The main idea of the proof is the following.
	Assume that the statement is false.
	Using the properness of $X$ one finds two subset $Y,Y' \in \mathcal Y$ that fellow-travel for a long distance.
	Thus we can find two ``geodesics'' respectively lying in $Z(Y)$ and $Z(Y')$ which stay far apart from each other, contradicting the hyperbolicity of $\dot X$.
	More precisely it works as follows.

	\paragraph{}
	Assume that the proposition is false.
	The first part of the proof takes place in $X$.
	In particular all the distances are measured with $\distV[X]$.
	Recall that $\mathcal Y$ contains finitely many $G$-orbits of subsets.
	Therefore there exists $Y \in \mathcal Y$ such that $K$ intersects infinitely many distinct translates of $Y$.
	In other words, for every $n \in \N$, there exist $g_n \in G$ and $y_n \in K$ such that $g_n^{-1}y_n$ belongs to $Y$ and for every $n \neq m$, $g_nY \neq g_mY$.
	By taking if necessary a subsequence, we can assume that $y_n$ converges to a point $y \in X$.
	
	\paragraph{}
	We first claim that $Y$ is unbounded.
	Suppose on the contrary that $Y$ is bounded.
	Since it is closed, $Y$ is compact.
	Without loss of generality we can assume that $g_n^{-1}y_n$ converges to a point $z$ of $Y$.
	By triangle inequlity
	\begin{equation*}
		\dist{g_nz}{g_mz}
		\leq \dist{g_nz}{y_n} + \dist {y_n}{y_m} + \dist {y_m}{g_mz}
		\leq \dist z{g_n^{-1}y_n} + \dist {y_n}{y_m} + \dist {g_m^{-1}y_m}z.
	\end{equation*}
	Thus for every $\eta >0$, if $n$ and $m$ are sufficiently large $\dist{g_nz}{g_mz} \leq \eta$.
	However the group $G$ acts properly on $X$.
	Consequently $(g_n)$ only takes finitely many values, which contradicts the fact that the elements of $(g_nY)$ are pairwise distinct.
	This completes the proof of our claim.
	
	\paragraph{}
	Let us now fix $d$ such that $\mu(e^{-L-\delta}d) > 4L + 12\delta$.
	Since $Y$ is rectifiably connected, for every $n \in \N$, there exists a point $y'_n$ in $g_nY$ such that $\dist {y_n}{y'_n} = d$.
	In particular $y'_n$ belongs to the $d$-neighborhood of $K$.
	Since $X$ is proper, this neighborhood is compact.
	Without loss of generality we can assume that $(y'_n)$ converges to a point $y'$.
	There exist two distinct integers $n,m$ such that $\dist {y_n}{y_m} \leq \delta$ and $\dist{y'_n}{y'_m} \leq \delta$.
	Moreover by construction $\dist{y_n}{y'_n} = \dist{y_m}{y'_m} = d$.
	From now on, the proof takes place in the cone-off $\dot X$.
	Let $x_n$ and $x'_n$ be the points of $Z(g_nY)$ defined by $x_n = (y_n,L+ \delta)$ and $x'_n=(y'_n,L + \delta)$.
	Similarly we defined two points $x_m$ and $x'_m$ in $Z(g_mY)$.
	Thus $\dist{x_n}{x_m} \leq 2L + 3\delta$ and  $\dist{x'_n}{x'_m} \leq 2L + 3\delta$.
	On the other hand
	\begin{equation*}
		\dist{x_n}{x'_n} \geq \mu\left(e^{-L-\delta}\dist{y_n}{y'_n} \right) \geq \mu\left(e^{-L-\delta}d\right) > 4L + 12\delta.
	\end{equation*}
	Thus there exists a point $p \in \dot X$ such that $\gro {x_n}{x'_n}p \leq \delta$ and $\min\{\dist {x_n}p, \dist{x'_n}p \}> 2L + 6\delta$
	According to \autoref{res: four point property} we get $\gro {x_m}{x'_m}p \leq 3\delta$.
	It follows then from \autoref{res: metric cone-off and horocone globally coincide} that $p$ should belong to $Z(g_nY)\setminus g_nY$ and $Z(g_mY)\setminus g_mY$.
	Thus $g_nY = g_mY$.
	Contradiction.
\end{proof}

\begin{coro}
\label{res: infinite stabilizer}
	Let $Y \in \mathcal Y$.
	The stabilizer $\stab Y$ acts co-compactly on $Y$.
	In particular, if $Y$ is unbounded then $\stab Y$ is infinite.
\end{coro}

\begin{proof}
	Since $G$ acts co-compactly on $X$ there exists a compact subset $K$ of $X$ such that $G\cdot K$ covers $X$.
	By \autoref{res: finitely many coset in a ball} only finitely many translates of $Y$ can intersect $K$.
	Thus there exists a finite subset $P$ of $G$ with the following property.
	For every $g \in G$ if $Y \cap gK \neq \emptyset$ then $g$ belongs to $\stab Y \cdot P$.
	Let us denote by $L$ the subset
	\begin{equation*}
		L = \bigcup_{h \in P} hK.
	\end{equation*}
	Since $P$ is finite, $L$ is a compact subset of $X$.
	We claim that $Y$ is contained in $\stab Y \cdot L$.
	Let $y$ be a point of $Y$.
	Since $K$ is a fundamental domain for the action of $G$ on $X$, there exist a point $x \in K$ and an element $g \in G$ such that $y = gx$.
	In particular, $Y \cap gK \neq \emptyset$.
	Thus $g$ can be written $g = uh$ with $u \in \stab Y$ and $h \in P$.
	Consequently, $y = u(hx)$ belongs to $\stab Y \cdot L$, which proves the claim.
	Recall that $Y$ is closed, thus $Y/\stab Y$ is compact.
	Hence $\stab Y$ acts co-compactly on $Y$.
\end{proof}

\begin{coro}
\label{res: cone-off proper}
	The cone-off $\dot X$ is proper and geodesic.
\end{coro}

\begin{proof}
	Let $(x_n)$ be a bounded sequence of points of $\dot X$.
	We want to prove that $(x_n)$ has an accumulation point.
	By \autoref{res: cone-off - bounded projection}, $(p(x_n))$ is a bounded sequence of points of $X$.
	According to \autoref{res: finitely many coset in a ball} there are only finitely many $Y \in \mathcal Y$ such that $Y$ contains a point of $(p(x_n))$.
	By taking if necessary a subsequence we can assume that one of the two following assertions holds.
	\begin{enumerate}
		\item For every $n \in \N$, $x_n$ belong to $X$.
		\item There exists $Y \in \mathcal Y$ such that for every $n \in \N$, $x_n$ belongs to $Z(Y)$.
	\end{enumerate}
	Assume first that all the points $x_n$ belong to $X$.
	By \autoref{res: metric cone-off and base} $(x_n)$ is bounded as a sequence of $X$.
	Since $X$ is proper, $(x_n)$ admits an accumulation point.
	However $X$ endowed with $\distV[X]$ or $\distV[\dot X]$ has the same topology.
	Therefore as a sequence of $\dot X$, $(x_n)$ also admits an accumulation point.
	Assume now that there exists $Y$ such that all the points $x_n$ belong to $Z(Y)$.
	For every $n \in \N$, we write $(y_n,r_n)$ for the point $x_n$.
	Since $(x_n)$ is bounded, $(y_n)$ is a bounded sequence of $Y \subset X$ and $(r_n)$ a bounded sequence of $\R_+$.
	By taking a subsequence if necessary we can assume that $(y_n)$ as a sequence of points of $X$ converges to a point $y$ and $(r_n)$ converges to a non-negative number $r$.
	Since $Y$ is closed $y$ is a point of $Y$.
	Thus $x = (y,r)$ defines a point of $Z(Y)$.
	Recall that we assumed $Y$ to be locally undistorted.
	Consequently $(y_n)$ converges to $y$ not only for the metric $\distV[X]$ but also for $\distV[Y]$.
	The map $Z(Y) \hookrightarrow \dot X$ being $1$-Lispchitz we get that for every $n \in \N$,
	\begin{equation*}
	 	 \cosh\dist[\dot X]{x_n}x  \leq \cosh\dist[Z(Y)] {x_n}x = \cosh(r_n-r) + \frac 12e^{-(r_n+r)}\dist[Y]{y_n}y^2.
	\end{equation*}
	Hence $(x_n)$ converges to $x$.
	Finally every bounded sequence of points of $\dot X$ admits an accumulation point, hence $\dot X$ is proper.
	On the other hand we know that $\dot X$ is a length space.
	By Hopf-Rinow Theorem, $\dot X$ is geodesic \cite[Chapter I.3, Theorem 3.7]{BriHae99}.
\end{proof}

\paragraph{}
The action of $G$ on $X$ naturally extends by homogeneity into an action of $G$ on $\dot X$.
If $x = (y,r)$ is a point of the cone $Z(Y)$ over $Y \in \mathcal Y$ and $g$ and element of $G$, then $gx$ is the point of $Z(gY)$ defined by $gx = (gy,r)$.

\begin{prop}
\label{res: proper action on the cone-off}
	The group $G$ acts properly on the cone-off $\dot X$.
\end{prop}

\begin{proof}
	Let $x$ be a point of $\dot X$ and $r$ a positive number.
	We denote by $P$ the set of elements $g \in G$ such that $g\ball xr \cap \ball xr \neq \emptyset$.
	According to \autoref{res: cone-off - bounded projection} the projection $p : \dot X \rightarrow X$ maps $\ball xr$ onto a bounded subset $S$ of $X$.
	By construction for every $g \in P$, $gS \cap S \neq \emptyset$.
	However, $X$ is proper and $G$ acts properly on it.
	Therefore $P$ is necessarily finite.
	Consequently, $G$ acts properly on $\dot X$.
\end{proof}

\begin{prop}
\label{res: parabolic points}
	Let $Y \in \mathcal Y$.
	We denote by $Z^-(Y)$ the set of all points $x = (y,r)$ of $Z(Y)$ with $r > L$.
	Then $Z^-(Y)$ is an open horoball centered at a point $\xi \in \partial \dot X$.
	Moreover the stabilizer of $\xi$ which is exactly $\stab Y$ does not contain any hyperbolic element of $G$.
\end{prop}

\begin{proof}
	Let us fix a point $y_0$ in $Y$.
	We denote by $x_0$ the point of $Z(Y)$ defined by $x_0 = (y_0, L)$.
	Let $\rho : \R_+ \rightarrow \dot X$ be the function that sends $t \in \R_+$ to the point $(y_0,L+t)$ of $Z(Y)$.
	By construction, $\rho$ is a geodesic ray of $Z(Y)$.
	Thus, by \autoref{res: metric cone-off and horocone globally coincide} it is also a geodesic ray of $\dot X$.
	Therefore it defines a point $\xi = \lim_{t \rightarrow + \infty} \rho(t)$ of $\partial \dot X$.
	Let $h : \dot X \rightarrow \R$ be the Buseman function associated to $\rho$.
	\begin{equation*}
		h(x) = \limsup_{t \rightarrow + \infty} \dist x{\rho(t)} -t.
	\end{equation*}
	Let $x$ be a point of $\dot X$.
	Assume first that $x$ does not belong to $Z^-(Y)$.
	In particular for every $t \in \R_+$, we have $\dist x{\rho(t)} \geq \dist {x_0}{\rho(t)} = t$.
	Consequently $h(x) \geq 0$.
	Assume now that $x$ is a point of $Z^-(Y)$ of the form $x = (y,r)$.
	In particular $r > L$.
	By \autoref{res: metric cone-off and horocone globally coincide} the metric of $Z(Y)$ and $\dot X$ coincide on $Z^-(Y)$.
	It follows from \autoref{res: horocone - horoball} that $h(x) = L - r < 0$.
	This exactly means that $Z^-(Y)$ is a horoball centered at $\xi$.
	
	\paragraph{}By construction every element of $\stab Y$ fixes $\xi$.
	Let us prove now the other inclusion.
	Let $g$ be an element of $G$ such that $g\xi = \xi$.
	By construction for every $t \in \R_+$, $g\rho(t)$ is defined to be the points of $Z(gY)$ given by $(gy,L+t)$.
	However, since $g\xi = \xi$, the Hausdorff distance between the geodesic rays $\rho$ and $g\rho$ is bounded.
	It forces $g$ to stabilizes $Y$.
	
	\paragraph{}Let $g$ be an element of $\stab Y$.
	Let $r \in [L, + \infty)$.
	We denote by $x_r$ the point of $Z(Y)$ given by $x_r = (y,r)$.
	According to \autoref{res: metric cone-off and horocone globally coincide} we have
	\begin{equation*}
		\cosh \dist {gx_r}{x_r} = 1 + \frac 12 e^{-2r}\dist {gy}y^2.
	\end{equation*}
	In particular $\lim_{r \rightarrow + \infty}\dist{gx_r}{x_r}=0$.
	It follows that $\len g = 0$ and thus $\len[stable] g = 0$.
	Consequently $\stab Y$ cannot contain a hyperbolic element of $G$.
\end{proof}

\begin{prop}
\label{res: co-compact action after removing horoballs}
	Let $X^+$ be the $L$-neighborhood of $X$ in $\dot X$.
	The action of $G$ on $X^+$ is co-compact.
\end{prop}

\begin{proof}
	The quotient space $X^+/G$ can be obtained by attaching on $X/G$ the sets $(Z(Y)\setminus Z^-(Y))/\stab Y$ for $Y \in \mathcal Y/G$.
	Since $G$ acts co-compactly on $X$, $X/G$ is compact.
	On the other hand, for every $Y \in \mathcal Y$, $\stab Y$ acts co-compactly on $Y$ (see \autoref{res: infinite stabilizer}).
	Consequently $(Z(Y)\setminus Z^-(Y))/\stab Y$ which is homeomorphic to  $(Y/\stab Y) \times \intval 0{L}$ is also compact.
	Recall that $\mathcal Y/G$ is finite.
	Hence $X^+/G$ is obtained by attaching together finitely many compact sets.
	Therefore it is compact.
\end{proof}

\begin{prop}
\label{res: cone-off - G relatively hyperbolic}
	Assume that every $Y \in \mathcal Y$ is unbounded.
	Let us identify $\mathcal Y/G$ with a set of representatives of the $G$-orbits of the elements of $\mathcal Y$.
	Then the group $G$ is hyperbolic relative to the collection $\set{\stab {Y}}{Y \in \mathcal Y/G}$.
\end{prop}

\begin{proof}
The proof recollects all the results of this section.
First it follows from \autoref{res: cone-off proper} and \autoref{res: proper action on the cone-off} that $\dot X$ is proper geodesic space and $G$ acts properly on it.
Let $Y \in \mathcal Y$.
According to \autoref{res: parabolic points}, $Z^-(Y)$ is an open horoball centered at a point $\xi$ of $\partial \dot X$.
Moreover since $Y$ is unbounded, the stabilizer of $\xi$ is infinite (\autoref{res: infinite stabilizer}) and does not contain any hyperbolic element (\autoref{res: parabolic points}).
By construction for every distinct $Y,Y' \in \mathcal Y$, $Z^-(Y)$ and $Z^-(Y')$ do not intersect.
Note that $X^+$ is exactly the space obtained by removing from $\dot X$ all the horoballs $Z^-(Y)$ where $Y \in \mathcal Y$.
By \autoref{res: co-compact action after removing horoballs}, the action of $G$ on $X^+$ is co-compact.
It follows from these observations that the group $G$ is relatively hyperbolic.
Moreover its maximal parabolic subgroups are exactly the stabilizers $\stab Y$ for $Y \in \mathcal Y$.
\end{proof}


\subsection{Proof of \autoref{res: main theorem weak geometric form}}
\label{sec: cone-off - proof}

\paragraph{}
In this section we complete the proof of \autoref{res: main theorem weak geometric form}.
To that end we consider the following data.
Let $G$ be a group.
Let $\epsilon >0$.
For every $n \in \N$, we are given
\begin{enumerate}
	\item a pointed proper length space $(X_n,e_n)$ on which $G$ acts properly co-compactly by isometries such that the diameter of $X_n/G$ is uniformly bounded.
	\item a $G$-invariant collection $\mathcal Y_n$ of closed unbounded locally undistorted subsets of $X_n$ such that $X_n$ is $\epsilon$-simply-connected relative to $\mathcal Y_n$ and  $\mathcal Y_n/G$ is finite.
\end{enumerate}
Let $\omega$ be a non-principal ultra-filter.
We assume that $(X_n,e_n)$ is asymptotically tree-graded with respect to $(\mathcal Y_n)$.
In addition, we suppose that the distortion of $(\mathcal Y_n)$ is $\omega$-controlled.

\begin{proof}[Proof of \autoref{res: main theorem weak geometric form}]
	For every $n \in \N$, we denote by $\dot X_n$ the cone-off of $X_n$ relative to $\mathcal Y_n$.
	Applying \autoref{res: tree-graded ultra-limit change of base point}, for every sequence of base points $(e'_n) \in \Pi_{n \in \N}X_n$, $(X_n,e'_n)$ is asymptotically tree-graded with respect to $(\mathcal Y_n)$.
	Thus we can apply \autoref{res: cone-off curvature general case}: there exists a subset $A$ of $\N$ with $\omega(A) = 1$ such that for every $n \in A$,  the cone-off $\dot X_n$ is $\delta$-hyperbolic where $\delta = 900 \boldsymbol\delta$.
	Let $n \in A$.
	Let us identify $\mathcal Y_n/G$ with a set of representatives of the $G$-orbits of the elements of $\mathcal Y_n$.
	We assumed that $X_n$ was proper and $G$ acts properly co-compactly by isometries on it.
	Moreover $\mathcal Y_n$ has finitely many $G$-orbits and all its elements are unbounded.
	It follows from \autoref{res: cone-off - G relatively hyperbolic} that $G$ is hyperbolic relative to the collection $\set{\stab {Y}}{Y \in \mathcal Y_n/G}$
\end{proof}


\section{Proof of \autoref{res: main theorem geometric form}}
\label{sec: spaces with a tree graded asymptotic cone}

This section is dedicated to the proof of  \autoref{res: main theorem geometric form}.
We start with the following definitions.

\begin{defi}
	An \emph{$\epsilon$-separated} set $A$ in $X$ is a set such that for every $x,y\in A$, $\dist xy\geq \epsilon$.
	A \emph{$\delta$-net} in $X$ is a set $B$ such that  every point of $X$ lies in the $\delta$-neighborhood of $B$.
	A \emph{$(\delta, \epsilon)$-net} is a $\delta$-net which is $\epsilon$-separated.
\end{defi}

\rem A maximal $\delta$-separated set is a net and will be referred to as a \emph{$\delta$-snet}.

\begin{defi}
	 We say that a subset $Y$ of $X$ has \emph{bounded geometry} is for every $x \in X$, for every $r \geq 0$, $Y \cap \ball xr$ is finite.
\end{defi}

\begin{lemm}
\label{res: snet}
	Let $G$ be a group acting properly cocompactly by isometries on a proper metric space $X$.
	For every $\delta>0$ there exists a $G$-invariant $\delta$-snet in $X$ with bounded geometry.
\end{lemm}

\begin{proof}
	Since $G$ acts properly and co-compactly by isometries on $X$, the distance between orbits in $X$ defines a bounded metric on $\bar X = X/G$, the space of $G$-orbits (see \cite[Proposition 8.5]{BriHae99}).
	Let $\delta >0$.
	We fix a $\delta$-snet $\bar S$ in $\bar X$.
	Let $S$ be the pre-image of $\bar S$ in $X$.
	By construction, $S$ is $G$-invariant $\delta$-net.
	Moreover it is $\delta$-separated.
	In particular $S$ is a closed subset of $X$.
	Since $X$ is proper the intersection of $S$ with any ball is finite.
	Thus $S$ has bounded geometry.
\end{proof}

\begin{proof}[Proof of \autoref{res: main theorem geometric form}]
	Let $G$ be a group acting properly co-compactly by isometries on a proper length space $X$.
	Let $\mathcal Y$ be a $G$-invariant collection of closed locally undistorted subsets of $X$ such that $X$ is $\epsilon$-simply-connected relative to $\mathcal Y$ for some $\epsilon >0$ and $\mathcal Y/G$ is finite.
	In particular there exists $\delta>0$ such that for every $Y \in \mathcal Y$, for every $y \in Y$ the natural embedding $Y \hookrightarrow X$ induces an isometry from $\ball y\delta$ onto its image.
	We assume that $X$ is sparsely asymptotically tree-graded with respect to $\mathcal Y$.
	Thus there exists a non-principal ultra-filter $\omega$, a sequence $(e_n)$ of points of $X$ and a sequence $(d_n)$ of real numbers  diverging to infinity with the following property.
	For every $n \in \N$ let $X_n = (1/d_n)X$ and $\mathcal Y_n = \{ (1/d_n) Y \mid Y\in \mathcal Y\}$.
	Then $(X_n,e_n)$ is sparsely tree-graded with respect to $(\mathcal Y_n)$.

\paragraph{}

 Since $\mathcal Y/ G$ is finite, there is a uniform bound on the diameter of the bounded sets of $\mathcal Y$.  Hence for any choice of $Y_n\in \mathcal Y_n$, $\limo Y_{n}$ is a point or $Y_n$ is unbounded \oas.  This implies that the sparsely tree-grading structure remains unchanged when only considering the unbounded sets in $\mathcal Y$.  Since $G$ acts properly, the stabilizer of any bounded subset must be finite.  Hence $G$ is hyperbolic relative to $\set{\stab Y}{Y \in \mathcal Y/G \text{ and } Y \text{ is unbounded}}$ if and only if  $G$ is hyperbolic relative to $\set{\stab Y}{Y \in \mathcal Y/G}$.  Thus we will assume that each element of $\mathcal Y$ (and hence $\mathcal Y_n$) is unbounded.
 
\paragraph{}

	We would like to apply \autoref{res: main theorem weak geometric form}, however the distortion of $(\mathcal Y_n)$ may not be $\omega$-controlled (see \autoref{def: distortion omega controlled}).
	To handle this difficulty we will substitute $X_n$ for a slightly large space $\tilde X_n$ which will fulfilled all the assumptions of \autoref{res: main theorem weak geometric form}.
	This construction will strongly use the fact that $(X_n,e_n)$ is sparsely tree-graded with respect to $(\mathcal Y_n)$.
	More precisely we proceed as follows.
		
	\paragraph{}
	Let $(\delta_n)$ be a sequence converging to $0$.
	Fix $\eta > 0$.
	Let $n \in \N$.
	According to \autoref{res: snet}, $X_n$ admits a $G$-invariant $\delta_n$-snet with bounded geometry.
	We write $S_n$ for the set of points in this $\delta_n$-snet which are $\delta_n$-close to some $Y \in \mathcal Y_n$.
	We denote by $T_n$ the set of unordered pairs $(x,y)$ of points of $S_n$ satisfying the folowing properties.
	\begin{itemize}
		\item There exists $Y \in \mathcal Y_n$ such that $x$ and $y$ lie in the $\delta_n$-neighborhood of $Y$.
		\item The distance $\dist xy$ is a most $\eta$.
	\end{itemize}
	For every such pair $(x,y) \in T_n$, we attach a metric arc $a_{x,y}$ to $X_n$ of length $\dist xy$ (measured with the distance of $X_n$) whose endpoints are $x$ and $y$.
	We denote by $\tilde X_n$ be the resultant space, i.e.
	\begin{equation*}
		\tilde X_n = X_n\sqcup \set{ a_{x,y} }{ (x,y)\in  T_n}/  \sim
	\end{equation*}
	where $\sim$ is the equivalence relation that identifies for every $(x,y) \in T_n$ the endpoints of $a_{x,y}$ respectively with $x$ and $y$.
	Then $\tilde X_n$ has an natural metric such that the quotient map from $X_n\sqcup \set{ a_{x,y} }{ (x,y)\in  T_n}$  to $\tilde X_n$ is an isometry when restricted to $X_n$ or $ \set{ a_{x,y} }{ (x,y)\in  T_n}$.
	We will generally identify $X_n$ and $ \set{ a_{x,y} }{ (x,y)\in  T_n}$ with their image in $\tilde X_n$.
	Recall that $X_n$ is proper and  $S_n$ has bounded geometry, therefore $\tilde X_n$ is a proper metric space.
	
	\paragraph{}
	We now define a collection $\tilde {\mathcal Y}_n$ of subsets of $\tilde X_n$ which is in one-to-one correspondence with $\mathcal Y_n$.
	If $Y$ is an element of $\mathcal Y_n$, then $\tilde Y$ is the union (in $\tilde X_n$) of $Y$, all the points of $S_n$ in the $\delta_n$-neighborhood of $Y$ and all the arcs $a_{x,y}$ such that $x$ and $y$ are two points lying in the $\delta_n$-neighborhood of $Y$ such that $\dist xy \leq \eta$.
	In other words
	\begin{equation*}
		\tilde Y = Y \cup \left(Y^{+\delta_n} \cap S_n\right) \cup \set{a_{x,y}}{(x,y) \in T_n, \text{ and } x,y \in Y^{+\delta_n}}.
	\end{equation*}
	In particular $Y$ isometrically embeds into $\tilde Y$.
	The collection $\tilde {\mathcal Y}_n$ is the family of all subsets $\tilde Y$ obtained in this way.
	The next lemmas investigate the asymptotical properties of the sequence $(\tilde X_n)$.
	
	\begin{lemm}
	\label{res: tilde X simply connected}
		The space $\tilde X_n$ is $\eta$-simply-connected relative to $\mathcal Y_n$ \oas.
	\end{lemm}
	
	\begin{proof}
		Let $n \in \N$.
		The space $\tilde X_n$ is obtained by attaching edges of length at most $\eta$ to $X_n$ which is geodesic.
		Therefore every loop $\gamma$ in $\tilde X_n$ is homotopic to a product $\gamma' \cdot \gamma_1 \cdots \gamma_\ell$ where $\gamma'$ is a loop in $X_n$ and the $\gamma_i$ are freely homotopic to a loop of diameter at most $\eta$.
		Thus the conclusion follows from the fact that $X_n$ is $\epsilon/d_n$-simply-connected relative to $\mathcal Y_n$.
	\end{proof}

	\begin{lemm}
	\label{res: tilde X tree-graded}
		The sequence $(\tilde X_n,e_n)$ is sparsely tree-graded with respect to $(\tilde{\mathcal Y}_n)$.
	\end{lemm}
	
	\rem
	Note first that the diameter of $\tilde X_n/G_n$ is uniformly bounded.
	Indeed for every $n \in \N$, $\diam(\tilde X_n/G) \leq \diam (X_n/G) + \eta$.
	In particular the $\omega$-limit of $\tilde X_n$ does not depend on the choice of the base points $(e_n)$.
	More precisely, according to \autoref{res: tree-graded ultra-limit change of base point}, whatever the sequence of base points $(e_n)$ is, $(\tilde X_n,e_n)$ is sparsely tree-graded with respect to $(\tilde{\mathcal Y}_n)$.
	
	\begin{proof}
		We have the following equalities.
		\begin{equation}
			\limo \tilde X_n
			= \limo X_n \cup \limo \set{a_{x,y} }{ (x,y)\in T_n} = \limo X_n \cup\set {a_{ x,  y} }{ ( x,   y) \in Q}
		\end{equation}
		and for every $(Y_n) \in \Pi_{n \in \N} \mathcal Y_n$,
		\begin{equation}
			\limo \tilde Y_n
			= \limo Y_n \cup \limo  \set{ a_{x,y} }{ (x,y)\in T_n \text{ and } x,y \in Y_n^{+\delta_n}}
			= \limo Y_{n} \cup \set{ a_{ x,  y} }{ ( x,   y) \in W_{(Y_n)}}
		\end{equation}
		where
		\begin{eqnarray*}
			Q & = & \set{(\limo x_n, \limo y_n) }{x_n,y_n \in X_n \text{ and } (x_n, y_n) \in T_n \text{ \oas} }, \\
			W_{(Y_n)} & = & \set{ (\limo x_n, \limo y_n) }{ x_n, y_n\in  Y_n^{+\delta_n} \text { and } \dist{x_n}{y_n} \leq \eta \text{ \oas}},
		\end{eqnarray*}
		and $a_{ x, y}$ is a metric arc of length $\dist xy$ attached to $\limo X_n$ in the same way we did for $X_n$.
		
		\paragraph{}We now look at Axiom $(T_1^\omega)$.
		Let $(\tilde Y_n)$ and  $(\tilde Y'_n)$ be two sequences of $\Pi_{n \in \N}\tilde{\mathcal Y}_n$ such that $\tilde Y_n \neq \tilde Y'_n$ \oas.
		It particular it implies that $Y_n \neq Y'_n$ \oas, where $Y_n$ and $Y'_n$ are the elements of $\mathcal Y_n$ from which $\tilde Y_n$ and  $\tilde Y'_n$ have been built.
		Consequently $\limo Y_n$ and $\limo Y'_n$ have at most one common point.
		Therefore it is sufficient to prove that  $\limo \tilde Y_n \cap \limo \tilde Y'_n$ is contained in $\limo Y_n \cap \limo Y'_n$.
		Let $z$ be a point of $\limo \tilde Y_n \cap \limo \tilde Y'_n$.
		Without loss of generality we can assume that $z$ does not belong to  $\limo Y_n\cap \limo Y'_n$.
		Consequently, for every $n \in \N$ there exist $x_n,y_n$ with $\dist{x_n}{y_n} \leq \eta$ which are in the $\delta_n$-neighborhood of both $Y_n$ and $Y'_n$ and a point $z_n$ on the arc $a_{x_n,y_n}$ such that $z$ is the limit $z = \limo z_n$.
		Since the sequence $(\delta_n)$ converges to $0$, $x=\limo x_n$ and $y=\limo y_n$ lie in the intersection of $\limo Y_n$ and $\limo Y'_n$, thus $x=y$.
		In other words $\limo \dist{x_n}{y_n} =0$.
		However by construction $z_n$ is on the arc between $x_n$ and $y_n$ whose length is $\dist{x_n}{y_n}$.
		At the limit we obtain $z = x = y$.
		Hence $z$ belongs to $\limo Y_n \cap \limo Y'_n$.
		Consequently Axiom $(T_1^\omega)$ holds.
		
		\paragraph{}		
		Only Axiom $(T_2^\omega)$ is left to prove.
		Let $\alpha : S^1 \to \limo \tilde X_n$ be a closed curve.
		Suppose that $\alpha$ is not contained in a single set of $\tilde{\mathcal Y} = \set{\limo \tilde Y_n}{ Y_n \in \mathcal Y_n}$.
		We denote by $U$ the set $\limo \tilde X_n \setminus \limo X_n$.
		In other words $U$ is the union for all $(x,y) \in Q$ of the open arcs $a_{x,y}\setminus\{x,y\}$.
		Notice that $U$ is open in $\limo \tilde X_n$.
		Let $\{J_k\}$ be the set  of disjoint open intervals of $\alpha^{-1}(U)$.
		Note that both endpoints of an interval $J_k$ must be mapped into the same element of $\mathcal Y $.
		Since the pieces of $\limo X_n$ are convex, we can then define a new closed loop $\gamma$ by replacing each subpath $\alpha(\bar J_k)$ by a geodesic contained in a single piece between its endpoints (here $\bar J_k$ stand for the closure of $J_k$).
		Then $\gamma$ is a closed curve in $\limo X_n$ which is not contained in a single piece of $\limo X_n$.
		By \autoref{M-transitions}, there exists distinct transition points $t_1,t_2\in S^1$ such that $\gamma(t_1) = \gamma(t_2)$.
		Notice that $\gamma(t) = \alpha(t)$ for all transition points.
		Thus $\alpha(t_1)= \alpha(t_2)$.
		In particular $\alpha$ is not a simple curve.
		This implies $(T_2^\omega)$ and completes the proof of the lemma.	
	\end{proof}

	\begin{lemm}
	\label{res: omega controlled distortion}
		The distortion of $(\tilde{\mathcal Y}_n)$ is $\omega$-controlled.
	\end{lemm}

	\begin{proof}
		Let $n \in \N$.
		Let $\tilde Y \in \tilde {\mathcal Y}_n$.
		Note that by construction $\tilde Y\cap S_n$ is a $\delta_n$-snet of $\tilde Y\cap X_n$.
		By adjoining points on the added arcs, we can extend $S_n$ into a subset $\tilde S_n$ of $\tilde X_n$ such that $\tilde Y\cap \tilde S_n$ is a $\delta_n$-snet of $\tilde Y$.
		Notice that any two points in $\tilde Y \cap \tilde S_n$ which are at most $\eta$ apart are still connected by a geodesic contained in $\tilde Y$.
		This extension can be done in such a way that every $\tilde Y \in \tilde{\mathcal Y}_n$ satisfies the same property.
		
		\paragraph{}
		Let $(\tilde Y_n)$ be a sequence of $\Pi_{n \in \N}\tilde{\mathcal Y}_n$.
		Let $(y_n)$ and $(y'_n)$ be two elements of $\Pi_{n\in \N}\tilde Y_n$.
		Our goal is to compare $\limo \dist[\tilde Y_n]{y_n}{y'_n}$ and $\limo \dist[\tilde X_n]{y_n}{y'_n}$.
		By definition of the length metric we have
		\begin{equation*}
			 \limo \dist[\tilde Y_n]{y_n}{ y'_n} \geq \limo \dist[\tilde X_n]{y_n}{ y'_n}.
		\end{equation*}
		Hence it is enough to show that if $\limo \dist[\tilde X_n]{y_n}{y'_n}$ is finite, then the reverse inequality holds.
		As we explained before, the choice of the base points does not affect the $\omega$-limit of $\tilde X_n$.
		Therefore we will work in $\limo (\tilde X_n,y_n)$ and put $y = \limo y_n$.
		Since $\limo \dist[\tilde X_n]{y_n}{y'_n}$ is finite, the sequence $(y'_n)$ defines a point $y' = \limo y'_n$ in $\limo (\tilde X_n,y_n)$.
		Fix a geodesic $\gamma: \intval ab \rightarrow \limo\tilde X_n$ from $ x$ to $ y$.
		Since $(\tilde X_n)$ is sparsely tree-graded with respect to $(\tilde {\mathcal Y}_n)$, $\gamma$ is contained in the piece $\limo\tilde Y_n$.
		Fix a partition $a=t_0<t_1<\dots< t_k =b$ such that for every $j \in \intvald 0{k-1}$,
		\begin{equation*}
			\dist{\gamma(t_j)}{ \gamma(t_{j+1})} < \eta/2.
		\end{equation*}

		\paragraph{}Recall that for every $n \in \N$, $\tilde Y_n \cap \tilde S_n$ is a $\delta_n$-snet of $\tilde Y_n$,
		Therefore for every $j \in \intvald 1{k-1}$, there exists a sequence $(z_{n,j}) \in \Pi_{n \in \N} \tilde Y_n \cap \tilde S_n$ such that $\gamma(t_j) = \limo z_{n,j}$.
		In addition for every $n \in \N$ there exists $z_{n,0}, z_{n,k} \in  \tilde Y_n \cap \tilde S_n$ such that $\dist[\tilde X_n]{y_n}{z_{n,0}} \leq \delta_n$ and $\dist[\tilde X_n]{y'_n}{z_{n,k}} \leq \delta_n$.
		In particular $\limo z_{n,0} = y$ and $\limo z_{n,k} = y'$.
		There exists a subset $A$ of $\N$ with $\omega(A)=1$ such that for every $n \in A$, for every $j \in \intvald 0{k-1}$,
		\begin{equation*}
			\dist[\tilde X_n]{z_{n,j}}{ z_{n,j+1}} <\eta.
		\end{equation*}
		By construction of $\tilde S_n$ it implies that
		\begin{equation*}
			\dist[\tilde X_n]{z_{n,j}}{ z_{n,j+1}}= \dist[\tilde Y_n]{z_{n,j}}{z_{n,j+1}}.
		\end{equation*}
 		Recall that every $Y \in \mathcal Y$ is locally undistorted in $X$.
		By choice of $(\delta_n)$ it follows that for every $n \in \N$,
		\begin{equation*}
			\dist[\tilde X_n]{y_n}{z_{n,0}}= \dist[\tilde Y_n]{y_n}{z_{n,0}} \quad \text{and} \quad \dist[\tilde X_n]{y'_n}{z_{n,k}}= \dist[\tilde Y_n]{y'_n}{z_{n,k}}.
		\end{equation*}
		Consequently, for every $n \in A$ we have the following inequality.
		\begin{eqnarray*}
			\dist[\tilde Y_{n}]{y_n}{y'_n}
			& \leq & \dist[\tilde Y_n]{y_n}{z_{n,0}}  + \sum_{j=0}^{k-1} \dist[\tilde Y_n]{z_{n,j}}{z_{n,j+1}} + \dist[\tilde Y_n]{z_{n,k}}{y'_n} \\
			& = & \dist[\tilde X_n]{y_n}{z_{n,0}}  + \sum_{j=0}^{k-1} \dist[\tilde X_n]{z_{n,j}}{z_{n,j+1}} + \dist[\tilde X_n]{z_{n,k}}{y'_n} .
		\end{eqnarray*}
		Thus after taking the $\omega$-limit we get
		\begin{equation*}
			\limo \dist [\tilde Y_n]{y_n}{y'_n}
			\leq \sum_{j=0}^{k-1} \limo \dist[\tilde Y_n]{z_{n,j}}{z_{n,j+1}}
			= \sum_{j=0}^{k-1} \dist{\gamma(t_j)}{ \gamma(t_{j+1})}
			= \dist y{y'}
			= \limo \dist [\tilde X_n]{y_n}{y'_n}.
		\end{equation*}
		Thus the distortion of $(\tilde{\mathcal Y}_n)$ is $\omega$-controlled.
	\end{proof}

	\begin{lemm}
	\label{res: same stabilizers}
		There exists a subset $A$ of $\N$ with $\omega(A)=1$ such for every $n \in A$, for every $Y \in \mathcal Y_n$, $\stab Y = \stab{\tilde Y}$.
	\end{lemm}

	\begin{proof}
		Assume that our assertion is false.
		There exists a sequence $(Y_n) \in \Pi_{n \in \N}\mathcal Y_n$ such that $\stab {Y_n} \neq \stab {\tilde Y_n}$ \oas.
		Let $n \in \N$.
		Recall that $\tilde Y_n$ is obtained from $Y_n$ by adding arcs between two points $x,y \in S_n$ in the $\delta_n$-neighborhood $Y_n$ which are at most $\eta$ apart.
		Since $S_n$ is $G$-invariant $\stab{Y_n}$ is contained in $\stab{\tilde Y_n}$.
		Consequently for every $n \in \N$ there exists $g_n \in \stab{\tilde Y_n}$ such that $g_n$ does not belong to $\stab {Y_n}$ \oas.
		We assumed that the elements in the collection $\mathcal Y$ of subsets of $X$ were rectifiably connected and  infinite.
		Consequently for every $n \in \N$ we can find two points $x_n, x'_n \in Y_n$ such that the distance $\dist {x_n}{x'_n}$ (measured in $X_n$) is $1$.
		Let $n \in \N$.
		Since $g_n$ stabilizes $\tilde Y_n$, $g_nx_n$ and $g_nx'_n$ both belongs to $X_n \cap \tilde Y_n$.
		Thus they are in the $\delta_n$-neighborhood of $Y_n$.
		Moreover they belong to $g_nY_n$.
		By construction, the sequence $(\delta_n)$ converges to zero.
		Consequently $\lim g_nx_n$ and $\limo g_nx'_n$ are two distinct points of $\lim X_n$ lying in $\limo Y_n \cap \limo g_nY_n$.
		However $(X_n)$ is sparsely tree-graded with respect to $(\mathcal Y_n)$.
		It follows from Axiom $(T_1^\omega)$ that $g_n$ belongs to $\stab {Y_n}$ \oas.
		Contradiction.
	\end{proof}

	We can now complete the proof of \autoref{res: main theorem geometric form}.
	By construction for every $n \in \N$, $G$ acts properly co-compactly on $\tilde X_n$ and $(\tilde {\mathcal Y}_n)$ is $G$-invariant collection of closed unbounded locally undistorted subsets of $\tilde X_n$.
	Moreover the diameter of $\tilde X_n/G$ is uniformly bounded above by $\eta$.
	I follows from \autoref{res: tilde X simply connected}, \autoref{res: tilde X tree-graded} and \autoref{res: omega controlled distortion} that $(\tilde X_n)$ and $(\tilde{\mathcal Y}_n)$ satisfies the assumption of \autoref{res: main theorem weak geometric form}.
	Consequently there exists a subset $A$ of $\N$ with $\omega(A)=1$ such that for every $n \in A$, $G$ is hyperbolic relative to $\set{\stab {\tilde Y}}{\tilde Y \in \tilde{\mathcal Y}_n/G}$ where $\tilde{\mathcal Y}_n/G$ is identified with a set of representatives of the $G$-orbits of the elements of $\tilde{\mathcal Y}_n$.
	According to \autoref{res: same stabilizers} the stabilizers of the element of $\tilde{\mathcal Y}_n$ are the same as the ones of the elements of $\mathcal Y_n$ \oas.
	By construction $\mathcal Y_n$ is just the collection $\mathcal Y$ viewed in the rescaled space $X_n$.
	Therefore $G$ is hyperbolic relative to $\set{\stab  Y}{Y \in \mathcal Y/G}$ where $\mathcal Y/G$ is identified with a set of representatives of the $G$-orbits of the elements of $\mathcal Y$.	
\end{proof} 


\section{Application and comments}
\label{sec: application and comments}

In this section we prove the theorem given in the introduction.
We discuss also some questions naturally arising from this work.
Given a finite set $S$, recall that $\mathbf F(S)$ stands for the free group generated by $S$.

\begin{defi}
\label{def: relative finite presentation}
	Let $G$ be a group and $\{H_1,\dots, H_m\}$ a collection of subgroups of $G$.
	We say that $G$ is \emph{finitely presented relative to $\{H_1,\dots, H_m\}$} if there exist a finite set $S$ and a finite subset $R$ of the free product $\mathbf F(S)*H_1*\dots*H_m$ such that $G$ is isomorphic to the quotient of $\mathbf F(S)*H_1*\dots*H_m$ by the normal subgroup generated by $R$.
\end{defi}

\begin{theo}
\label{res: main theorem}
	Let $G$ be a finitely generated group and $\{H_1,\dots, H_m\}$ be a collection of subgroups of $G$.
	Assume that $G$ is finitely presented relative to $\{H_1,\dots, H_m\}$ and sparsely asymptotically tree-graded with respect to $\{H_1,\dots, H_m\}$.
	Then $G$ is hyperbolic relative to $\{H_1,\dots, H_m\}$.
\end{theo}

\begin{proof}
	Since $G$ is finitely presented relative to $\{H_1,\dots, H_m\}$, each $H_i$ is finitely generated.
	For every $i \in \intvald 1m$ we fix a presentation $\langle S_i | R_i\rangle$ of $H_i$ where $S_i$ is finite (note that $R_i$ might be infinite).
	We denote by $Y_i$ the Cayley graph associated to this presentation.
	The group $G$ being relatively finitely presented, it admits a presentation $\langle S\cup S_1\cup \dots \cup S_m |R \cup R_1 \cup \dots \cup R_m\rangle$ where $S$ and $R$ are finite.
	The space $X$ stands for the Cayley graph associated to this presentation.
	This space is quasi-isometric to $G$.
	Note that every $Y_i$ naturally embeds in $X$ and the stabilizer of $Y_i$ is $H_i$.
	Since the generating set of the presentation of $G$ is finite, $X$ is a proper geodesic space.
	Moreover each $Y_i$ is locally undistorted.
	We denote by $\mathcal Y$ the set of all translates of $Y_i$, i.e.
	\begin{equation*}
		\mathcal Y =  \bigcup_{i =1}^m \set{gY_i}{g \in G/H_i},
	\end{equation*}
	It follows from our asumptions that $X$ is sparsely asymptotically tree-graded with respect to $\mathcal Y$.
	By construction, any loop in $X$ is homotopic to a product of loops $\gamma_1 \cdots \gamma_\ell$ such that for every $j \in \intval 1\ell$ $\gamma_j$ is freely homotopic to a loop contained in one element of $\mathcal Y$ or representing an element of $R$.
	However $R$ is finite.
	Thus $X$ is $\epsilon$-simply-connected relative to $\mathcal Y$ for some $\epsilon >0$.
	Therefore by \autoref{res: main theorem geometric form}, $G$ is hyperbolic relative to $\{H_1, \dots, H_m\}$.
\end{proof}

\paragraph{Relatively finite presentation.}
The assumptions of \autoref{res: main theorem} cannot be weakened.
In particular it is really essential to assume that $G$ is finitely presented relative to $\{H_1, \dots, H_m\}$.
Consider for instance a lacunary hyperbolic group.
A finitely generated group is lacunary hyperbolic if one of its asymptotic cones is a tree.
A.~Ol’shanski\u\i, D.~Osin, and M.~Sapir showed the existence of lacunary hyperbolic groups which are not hyperbolic \cite{OlcOsiSap09}.
Recall that a group is hyperbolic relative to the trivial subgroup if and only if it is hyperbolic.
If $G$ is lacunary hyperbolic, then $G$ has an asymptotic cone which is a tree.
Hence $G$ is sparsely asymptotically tree-graded with respect to the trivial subgroup.
Hence the existence of lacunary hyperbolic groups which are not hyperbolic shows that it is still not enough to be sparsely asymptotically tree-graded with respect to a finitely generated subgroup.

\paragraph{}
Let us sketch now the construction of another example.
We build a group $G$ which is sparsely asymptotically tree-graded with respect to an infinite subgroup $H$.
However the asymptotic cone involved in this example is not a tree.
To that end we use the small cancellation theory.
We refer the reader to the book of R.~Lyndon and P.~Schupp for an exposition of the small cancellation theory \cite[Chapter V]{LynSch77}.
Fix an alphabet $\mathcal A$ and let $\{w_n\}$ be a set of cyclically reduced words satisfying the $C'(\lambda)$ assumption for $\lambda \ll 1$ and $\lim_{n\rightarrow \infty} |w_n|/|w_{n+1}|=0$.
Here $|w_n|$ stands for the length the word $w_n$.
 For example, we can take the alphabet $\mathcal A=\{a,b\}$ and the set $\{w_n\}_{n \geq 1000}$ where
 \begin{equation*}
 	w_n = \left(a^{2^{2^n}}b^{2^{2^n}}\right)^n.
 \end{equation*}
Given a subset $I$ of $\N$ we write $\mathcal R_I $ for the set $\{w_n^n \mid n\in I\}$.
We define the group $G_I$ by the presentation $G_I =  \pres{S}{\mathcal R_I}$ whereas $H_I$ is the subgroup of $G_I$ generated by $\{w_n\mid n\in I\}$.



\begin{prop}
\label{example1}
There exists an infinite subset $I$ of $\N$ such that $G_I$ is sparcely asymptotically tree-graded with respect to the set of distinct left cosets of $H_I$.
\end{prop}

To prove the proposition we will need the following two lemmas which are consequences of small cancelation theory.
The proof is left to the reader.

\begin{lemm}
\label{res: satisfying C'}
	For any $I\subset \N$, $\mathcal R_{I}$ is $C'(\mu)$ where $\mu \leq 4\min\{\lambda, 1/\min\{I\}\}$.
 \end{lemm}




\begin{lemm}
\label{unique limits}
	Let $I$ be a subset of $\N$.
	Let $g \in G_I\setminus H_I$.
	If $a_1,a_2\in H_I$ and $b_1,b_2\in gH_I$, then
	\begin{equation*}
		\dist[G_I]{a_1}{a_2}\leq 1000\lambda\left(\dist[G_I]{a_1}{b_1}+\dist[G_I]{a_2}{b_2}\right)
	\end{equation*}
\end{lemm}

\begin{proof}[Sketch of proof of \autoref{example1}]
The proof uses many results from the work of A.~Ol’shanski\u\i, D.~Osin, and M.~Sapir on lacunary hyperbolic groups  \cite{OlcOsiSap09}.
We put $I$ the set of integers larger than $40$.
The \autoref{res: satisfying C'} shows that $\mathcal R_I$ satisfies the conditions $(SC_1), (SC_2)$ stated in  \cite[Lemma 4.7]{OlcOsiSap09} for $\lambda \leq 1/100$.
Then \cite[Corollary 4.15]{OlcOsiSap09} implies that there exists an infinite subset $J$ of $I$ such that $G_J = \pres{S}{\mathcal R_J}$ which satisfies the graded small cancelation condition.
Let $\mathcal L$ be the set of all loops in the Cayley graph $\Gamma(G_J, S)$ of $G_J$ labeled by words in $\mathcal R_J$.
Then \cite[Theorem 4.17]{OlcOsiSap09} implies that there exists a sequence $d = (d_n)$ such that $\ascong{G_J}d$ is a circle-tree.
Moreover $\ascong{G_J}d$ is tree-graded with respect to the collection  $\mathcal C$ of all distinct limits of \emph{asymptotically visible} sequences in $\mathcal L$.
Thus we need only show that given a sequence $(g_n)$ of elements of $G_J$, if $\limo g_nH_J = \limo H_J$ then $g_n \in H_J$ \oas.
This follows from \autoref{unique limits}.
\end{proof}

\rem
The group $G_J$ from \autoref{example1} cannot be hyperbolic relative to $H_J$ since $H_J$ is not finitely generated.
Note that in this example $\ascong{G_J}d$ is tree-graded with respect to a family of circles.
Hence being sparsely asymptotically tree-graded with respect to a $H_J$ is definitely insufficient to imply hyperbolic relative $H_J$.
While $G_J$ is not hyperbolic relative to $H_J$, it is unknown if $G_J$ hyperbolic  relative to some other subgroup.

\paragraph{Control of the pieces.}
In the statement of \autoref{res: main theorem}, it is stronger to ask that $G$ is sparsely asymptotically tree-graded with respect to $\{H_1, \dots, H_m\}$ than just some asymptotic cone of $G$ is a tree-graded space.
Indeed the first assumption explains where the pieces of the tree-graded structure come from.
One could also ask whether the it is necessary for the pieces to be limits of a subgroup.
In \cite{Behrstock:2006gn}, J. Behrstock proved that any asymptotic cone of the mapping class group $\operatorname{MCG}(\Sigma)$ of a surface $\Sigma$ has global cut points and therefore is tree-graded (with respect to a non-trivial collection of pieces) but the pieces are not limits of subgroups.
Indeed $\operatorname{MCG}(\Sigma)$ is not hyperbolic relative to any family of proper subgroups \cite{Behrstock:2009bb,Anderson:2007cm}.
(It is weaky relatively hyperbolic in a certain sense, though.)
Therefore we wonder if there is a way to characterize finitely generated groups such that some of there asymptotic cones are tree-graded.
In particular does such a group have an acylindrical action on a hyperbolic space?
(See \cite{Osin:2013te} for a study of acylindrically hyperbolic groups.)

\paragraph{Lacunary relatively hyperbolic groups.}
In \cite{OlcOsiSap09}, A.Y.~Ol'shanski\u\i, D.~Osin and M.~Sapir used asymptotic geometry to extend the notion of hyperbolic group: a group is \emph{lacunary hyperbolic} if one of its asymptotic cones is an $\R$-tree.
Among others, they provided many examples of lacunary hyperbolic groups and proved that they share some common properties.
For instance, a lacunary hyperbolic group cannot contain a subgroup isomorphic to $\Z^2$ or the lamplighter group.
In addition, they proved that lacunary hyperbolic groups can be characterized as a limit of hyperbolic groups $G_0 \twoheadrightarrow G_1 \twoheadrightarrow G_2 \twoheadrightarrow \dots$ with some control of the hyperbolicity constant of $G_k$ compare to the injectivity radius of the map $G_k \twoheadrightarrow G_{k+1}$.

\paragraph{}
Following the same approach, one could introduce a new class of groups.
Let $G$ be a finitely generated group and $\{H_1, \dots, H_m\}$ a collection of subgroups of $G$.
We would say that $G$ is \emph{lacunary hyperbolic relative to} $\{H_1, \dots ,H_m\}$, if $G$ is sparsely asymptotically tree-graded with respect to $\{H_1,\dots, H_m\}$.
The first steps to study this class would be to solve the following questions.
\begin{itemize}
	\item Which groups are lacunary relatively hyperbolic but not lacunary hyperbolic?
	\item Is there a characterization of lacunary relatively hyperbolic groups as a limit of relatively hyperbolic groups?
	\item What are the common properties of lacunary relatively hyperbolic groups?
	Let $G$ be a lacunary hyperbolic group relative to $\{H_1, \dots, H_m\}$.
	If $G$ contains a subgroup isomorphic to $\Z^2$ is it necessarily conjugate to a subgroup of one of the $H_i$?
\end{itemize}

\makebiblio

\noindent
\emph{R\'emi Coulon} \\
Department of Mathematics, Vanderbilt University\\
Stevenson Center 1326, Nashville TN 37240, USA\\
\texttt{remi.coulon@vanderbilt.edu} \\
\texttt{http://www.math.vanderbilt.edu/$\sim$coulonrb/}

\vskip 5mm

\noindent
\emph{Michael Hull} \\
Department of Mathematics, Statistics, and Computer Science, University of Illinois at Chicago \\
322 Science and Engineering Offices (M/C 249) 851 S. Morgan Street\\
Chicago, IL 60607-7045\\
\texttt{mbhull@uic.edu} \\
\texttt{http://www.math.uic.edu/graduate/people/profile?netid=mbhull}

\vskip 5mm

\noindent
\emph{Curtis Kent} \\
Department of Mathematics, University of Toronto \\
Room 6290, 40 St. George Street \\
Toronto, Ontario, Canada M5S 2E\\
\texttt{curtkent@math.utoronto.ca} \\
\texttt{http://www.math.toronto.edu/cms/kent-curtis/}


\begin{thebibliography}{10}

\bibitem{Anderson:2007cm}
J.~W. Anderson, J.~Aramayona, and K.~J. Shackleton.
\newblock {An obstruction to the strong relative hyperbolicity of a group}.
\newblock {\em Journal of Group Theory}, 10(6):749--756, 2007.

\bibitem{ArzDel08}
G.~N. Arzhantseva and T.~Delzant.
\newblock {Examples of random groups.}
\newblock Nov. 2008.

\bibitem{Behrstock:2009bb}
J.~Behrstock, C.~Dru{\c t}u, and L.~Mosher.
\newblock {Thick metric spaces, relative hyperbolicity, and quasi-isometric
  rigidity}.
\newblock {\em Mathematische Annalen}, 344(3):543--595, 2009.

\bibitem{Behrstock:2006gn}
J.~A. Behrstock.
\newblock {Asymptotic geometry of the mapping class group and Teichm\"uller
  space}.
\newblock {\em Geometry {\&} Topology}, 10:1523--1578, 2006.

\bibitem{Bou71}
N.~Bourbaki.
\newblock {\em {\'El\'ements de math\'ematique. Topologie g\'en\'erale.
  Chapitres 1 \`a 4}}.
\newblock Hermann, Paris, 1971.

\bibitem{Bowditch:1991wl}
B.~H. Bowditch.
\newblock {Notes on Gromov's hyperbolicity criterion for path-metric spaces}.
\newblock In {\em Group theory from a geometrical viewpoint (Trieste, 1990)},
  pages 64--167. World Sci. Publ., River Edge, NJ, 1991.

\bibitem{Bowditch:2012ga}
B.~H. Bowditch.
\newblock {Relatively hyperbolic groups}.
\newblock {\em International Journal of Algebra and Computation},
  22(3):1250016--1250066, 2012.

\bibitem{BriHae99}
M.~R. Bridson and A.~Haefliger.
\newblock {\em {Metric spaces of non-positive curvature}}, volume 319 of {\em
  Grundlehren der Mathematischen Wissenschaften [Fundamental Principles of
  Mathematical Sciences]}.
\newblock Springer-Verlag, Berlin, 1999.

\bibitem{CooDelPap90}
M.~Coornaert, T.~Delzant, and A.~Papadopoulos.
\newblock {\em {G\'eom\'etrie et th\'eorie des groupes}}, volume 1441 of {\em
  Lecture Notes in Mathematics}.
\newblock Springer-Verlag, Berlin, 1990.

\bibitem{Coulon:il}
R.~Coulon.
\newblock {Asphericity and small cancellation theory for rotation families of
  groups}.
\newblock {\em Groups, Geometry, and Dynamics}, 5(4):729--765, 2011.

\bibitem{Coulon:2013tx}
R.~Coulon.
\newblock {Small cancellation theory and Burnside problem}.
\newblock {\em arXiv.org}, (1302.6933v2), Feb. 2013.

\bibitem{DelGro08}
T.~Delzant and M.~Gromov.
\newblock {Courbure m{\'e}soscopique et th{\'e}orie de la toute petite
  simplification}.
\newblock {\em Journal of Topology}, 1(4):804--836, 2008.

\bibitem{Dru02}
C.~Dru{\c t}u.
\newblock {Quasi-isometry invariants and asymptotic cones}.
\newblock {\em International Journal of Algebra and Computation},
  12(1-2):99--135, 2002.

\bibitem{DruSap05}
C.~Dru{\c t}u and M.~Sapir.
\newblock {Tree-graded spaces and asymptotic cones of groups}.
\newblock {\em Topology. An International Journal of Mathematics},
  44(5):959--1058, 2005.

\bibitem{Far98}
B.~Farb.
\newblock {Relatively hyperbolic groups}.
\newblock {\em Geometric and Functional Analysis}, 8(5):810--840, 1998.

\bibitem{Gromov:1981ve}
M.~Gromov.
\newblock {Groups of polynomial growth and expanding maps}.
\newblock {\em Publications Math{\'e}matiques. Institut de Hautes {\'E}tudes
  Scientifiques}, (53):53--73, 1981.

\bibitem{Gro87}
M.~Gromov.
\newblock {Hyperbolic groups}.
\newblock In {\em Essays in group theory}, pages 75--263. Springer, New York,
  1987.

\bibitem{Gro93}
M.~Gromov.
\newblock {Asymptotic invariants of infinite groups}.
\newblock In {\em Geometric group theory, Vol.\ 2 (Sussex, 1991)}, pages
  1--295. Cambridge Univ. Press, Cambridge, 1993.

\bibitem{Gro01b}
M.~Gromov.
\newblock {Mesoscopic curvature and hyperbolicity}.
\newblock In {\em Global differential geometry: the mathematical legacy of
  Alfred Gray (Bilbao, 2000)}, pages 58--69. Amer. Math. Soc., Providence, RI,
  2001.

\bibitem{Gro03}
M.~Gromov.
\newblock {Random walk in random groups}.
\newblock {\em Geometric and Functional Analysis}, 13(1):73--146, 2003.

\bibitem{Groves:2008ip}
D.~Groves and J.~F. Manning.
\newblock {Dehn filling in relatively hyperbolic groups}.
\newblock {\em Israel Journal of Mathematics}, 168:317--429, 2008.

\bibitem{Hruska:2010iw}
G.~C. Hruska.
\newblock {Relative hyperbolicity and relative quasiconvexity for countable
  groups}.
\newblock {\em Algebraic {\&} Geometric Topology}, 10(3):1807--1856, 2010.

\bibitem{LynSch77}
R.~C. Lyndon and P.~E. Schupp.
\newblock {\em {Combinatorial group theory}}.
\newblock Springer-Verlag, Berlin, 1977.

\bibitem{Olc80}
A.~Y. Ol'shanski{\u \i}.
\newblock {An infinite group with subgroups of prime orders}.
\newblock {\em Izvestiya Akademii Nauk SSSR. Seriya Matematicheskaya},
  44(2):309--321, 479, 1980.

\bibitem{OlcOsiSap09}
A.~Y. Ol'shanski{\u \i}, M.~Sapir, and D.~V. Osin.
\newblock {Lacunary hyperbolic groups}.
\newblock {\em Geometry {\&} Topology}, 13(4):2051--2140, 2009.

\bibitem{Osin:2006cf}
D.~V. Osin.
\newblock {Relatively hyperbolic groups: intrinsic geometry, algebraic
  properties, and algorithmic problems}.
\newblock {\em Memoirs of the American Mathematical Society}, 179(843):vi--100,
  2006.

\bibitem{Osin:2013te}
D.~V. Osin.
\newblock {Acylindrically hyperbolic groups}.
\newblock {\em arXiv.org}, Apr. 2013.

\bibitem{Pansu:1983bv}
P.~Pansu.
\newblock {Croissance des boules et des g\'eod\'esiques ferm\'ees dans les
  nilvari\'et\'es}.
\newblock {\em Ergodic Theory and Dynamical Systems}, 3(3):415--445, 1983.

\bibitem{Papasoglu:1996uv}
P.~Papasoglu.
\newblock {An algorithm detecting hyperbolicity}.
\newblock In {\em Geometric and computational perspectives on infinite groups
  (Minneapolis, MN and New Brunswick, NJ, 1994)}, pages 193--200. Amer. Math.
  Soc., Providence, RI, 1996.

\bibitem{Stallings:1971up}
J.~Stallings.
\newblock {\em {Group theory and three-dimensional manifolds}}.
\newblock Yale University Press, New Haven, Conn., 1971.

\bibitem{Stallings:1968us}
J.~R. Stallings.
\newblock {On torsion-free groups with infinitely many ends}.
\newblock {\em Annals of Mathematics. Second Series}, 88:312--334, 1968.

\bibitem{DriWil84}
L.~van~den Dries and A.~J. Wilkie.
\newblock {Gromov's theorem on groups of polynomial growth and elementary
  logic}.
\newblock {\em Journal of Algebra}, 89(2):349--374, 1984.

\bibitem{Yaman:2004ka}
A.~Yaman.
\newblock {A topological characterisation of relatively hyperbolic groups}.
\newblock {\em Journal f{\"u}r die Reine und Angewandte Mathematik. [Crelle's
  Journal]}, 566:41--89, 2004.

\end{thebibliography}
\end{document}